\documentclass[11pt]{amsart}

\title[Fundamental groups of the complements of shadows]{Presentation of the fundamental groups of complements of shadows}

\author[Ishikawa]{Masaharu Ishikawa}
\address{Department of Mathematics, Hiyoshi Campus, Keio University, 
4-1-1 Hiyoshi, Kohoku, Yokohama 223-8521, Japan}
\email{ishikawa@keio.jp}

\author[Koda]{Yuya Koda}

\address{
Department of Mathematics, Hiroshima University, 
1-3-1 Kagamiyama, Higashi-Hiroshima, 739-8526, Japan}
\email{ykoda@hiroshima-u.ac.jp}

\author[Naoe]{Hironobu Naoe}
\address{Department of Mathematics, Chuo University, 
1-13-27 Kasuga, Bunkyo-ku, Tokyo, 112-8551, Japan}
\email{naoe@math.chuo-u.ac.jp}

\dedicatory{Dedicated to Professor Norbert A'Campo on his 80th birthday}

\usepackage{amsfonts,amsmath,amssymb,amscd}
\usepackage{amsthm}
\usepackage{latexsym}
\usepackage{graphicx}
\usepackage{psfrag}
\usepackage{color}
\usepackage{ulem}

\usepackage{fancybox, ascmac}
\usepackage{multicol}

\theoremstyle{plain}
\newtheorem*{theorem*}{Theorem}
\newtheorem*{lemma*} {Lemma}
\newtheorem*{corollary*} {Corollary}
\newtheorem*{proposition*}{Proposition}
\newtheorem*{conjecture*}{Conjecture}
\newtheorem{theorem}{Theorem}[section]
\newtheorem{lemma}[theorem]{Lemma}

\theoremstyle{remark}

\newtheorem{definition}[theorem]{Definition}
\newtheorem{example}[theorem]{Example}
\newtheorem{remark}[theorem]{Remark}

\newtheorem*{example*}{Example}

\theoremstyle{definition}

\newtheoremstyle{citing}
  {}
  {}
  {\itshape}
  {}
  {\bfseries}
  {.}
  {.5em}
  {\thmnote{#3}}

\theoremstyle{citing}

\newcommand{\Integer}{\mathbb{Z}}
\newcommand{\Real}{\mathbb{R}}
\newcommand{\Complex}{\mathbb{C}}

\newcommand{\Int}{\mathrm{Int\,}}

\newcommand{\Nbd}{\mathrm{Nbd}}

\newcommand{\gl}{\frak{gl}}
\newcommand{\Sing}{\text{\rm Sing}}

\makeatletter
\@addtoreset{equation}{section}

\makeatother

\begin{document}

\begin{abstract}
A shadowed polyhedron is a simple polyhedron equipped with half integers on regions, called gleams, which represents a compact, oriented, smooth $4$-manifold.
The polyhedron is embedded in the $4$-manifold and it is called a {\it shadow} of that manifold.
A subpolyhedron of a shadow represents a possibly singular subsurface in the $4$-manifold. 
In this paper, we focus on contractible shadows obtained from the unit disk by attaching annuli along generically immersed closed curves on the disk.
In this case, the $4$-manifold is always a $4$-ball.
Milnor fibers of plane curve singularities and complexified real line arrangements can be represented in this way.
We give a presentation of the fundamental group of the complement of a subpolyhedron of such a shadow in the $4$-ball.
The method is very similar to the Wirtinger presentation of links in knot theory.
\end{abstract}

\maketitle

\section{Introduction}

The Milner fibration~\cite{Mil68} of singularities of a complex polynomial map is an important tool widely used when we  study the structure of singularities. In particular, the study of monodromy of the fibration plays an important role in understanding singularities.
In the case of polynomials of two variables, the plane curve given by a polynomial forms a 
real $2$-dimensional object embedded in $\Complex^2$. Thus, in that case, we can explain its properties more visually, and hence explicitly.
For example, using real deformations of complex plane curve singularities introduced by N. A'Campo~\cite{AC75a, AC75b}  and S. M.~Gusein-Zade~\cite{GZ74a, GZ74b, GZ77}, we can see the configurations of vanishing cycles of an isolated plane curve singularity from a diagram consisting 
of immersed intervals on $\Real^2$. Later, in~\cite{AC99,AC98}, A'Campo gave a way for restoring the Milner fibration from the diagram by replacing the real plane with the unit disk and regarding $\Complex^2$ as the tangent bundle of $\Real^2$. This method can also be applied to any generically immersed intervals and circles on the unit disk even if it cannot be obtained as a real deformation of a plane curve singularity. Such a diagram is called a {\it divide}. 
It is shown that the fibration associated with
a divide is equivalent to the Milnor fibration if it is obtained from
a real deformation of a plane curve singularity.

A {\it shadowed polyhedron} is a simple polyhedron equipped with half integers on regions, called gleams. It represents a compact, oriented, smooth $4$-manifold, in which that polyhedron is embedded in a natural way as a {\it shadow}~\cite{Tur92, Tur94}.
The union of a Milnor fiber of a plane curve singularity and the disks bounded by its vanishing cycles constitutes a polyhedron embedded in the Milnor ball, which is a $4$-ball in $\Complex^2$ centered at the singular point.
Topologically, this polyhedron is the union of the unit disk and a finite number of annuli attached along immersed curves on the disk so that the polyhedron is simple and contractible. 
In~\cite{IN20}, we regarded it as a shadow and determined its gleams in more general context, including oriented divides introduced by W. Gibson and the first author~\cite{GI02a}.

In this paper, we explain how to calculate the fundamental group of the complement of a subpolyhedron of a shadowed polyhedron in its $4$-manifold when the shadow consists of the unit disk and annuli attached along immersed curves on the disk so that the polyhedron is simple and contractible. 
As explained above, this includes the polyhedron of the fibration of a divide, and, as its special cases, 
it includes polyhedrons of Milnor fibrations and complexified real line arrangements.
Our main theorem is Theorem~\ref{thm1}, where a presentation of the fundamental group of the complement is given. The proof is very basic. We only use the van Kampen theorem. 
We can simplify the algorithm to obtain the presentation slightly when the subpolyhedron does not contain the region containing the boundary of the unit disk, which is explained in Theorem~\ref{thm2}.
In Section~4, it is explained that 
a Wirtinger presentation of the link group of a link in $S^3$ is obtained from our presentation  
when the gleams satisfy a certain condition. In Sections~5 and~6, we show some calculation of the fundamental groups of divides and complexified real line arrangements by using our method. 

This work is a continuation of the study of the relation between divides and shadows suggested by Norbert A'Campo. The authors would like to thank him for introducing them to this new world.
The first author is supported by JSPS KAKENHI Grant Numbers JP19K03499,
JP17H06128.
The second author is supported by JSPS KAKENHI Grant Numbers
JP20K03588, JP20K03614, JP21H00978.
The third author is supported by JSPS KAKENHI Grant Number JP20K14316. 
This work is supported by JSPS-VAST Joint Research Program, Grant number JPJSBP120219602.

\section{Preliminaries}

Throughout this paper, for a smooth manifold or a polyhedral space $A$, $\Int A$ denotes the interior of $A$,
$\partial A$ denotes the boundary of $A$, and $\Nbd(B;A)$ denotes a closed regular neighborhood of a subspace $B$ of $A$ in $A$,
where $A$ is equipped with the natural PL structure if $A$ is a smooth manifold.
A tree means a simply-connected graph.

\subsection{Shadowed polyhedron}

A compact space $X$ is called a \textit{simple polyhedron} if each point of $X$ has a neighborhood homeomorphic to one of (i)-(v) in Figure~\ref{fig17}.
The set of points of 
type 
(ii), (iii) or (v) is called the \textit{singular set} of $X$ 
and denoted by $\Sing(X)$. 
A point of type (iii) is a \textit{vertex}, and each connected component of $\Sing(X)$ 
with vertices removed is called an \textit{edge}.
Each connected component of $X\setminus\Sing(X)$ is called a \textit{region}. 
Hence a region consists of points of type (i) or (iv). 
A region is said to be \textit{internal} if it contains no points of type (iv).
The set of points of type (iv) or (v) is called the \textit{boundary of $X$} and denoted by $\partial X$.

\begin{figure}[htbp]
\includegraphics[scale=0.6, bb=129 541 543 714]{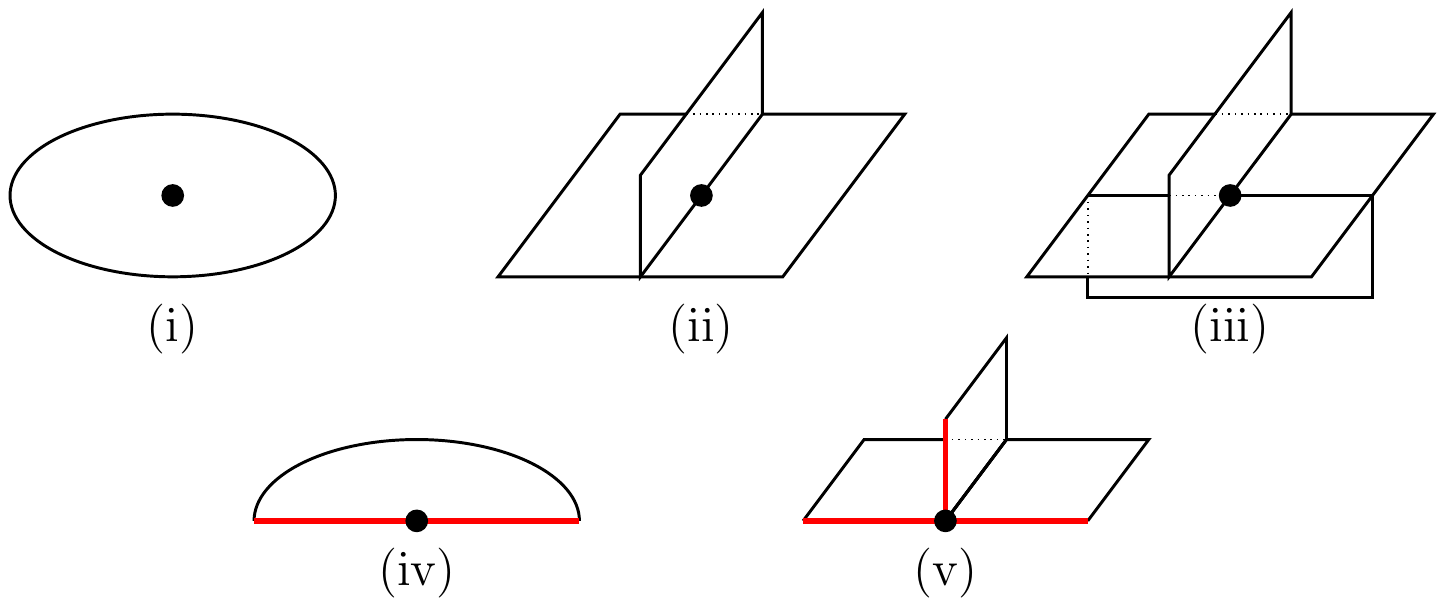}
\caption{The local models of a simple polyhedron.}
\label{fig17}
\end{figure}

Let $W$ be a compact, oriented, smooth $4$-manifold with boundary $\partial W$, and $L$ be a link in $\partial W$.
A polyhedron $X$ embedded in $W$ is said to be {\it locally flat} if, for each $x\in X$, $\Nbd(x;X)$ is embedded in a smooth $3$-ball in $\Nbd(x;W)$. 
If a handle decomposition of $W$ contains neither $3$- nor $4$-handles, $W$ collapses onto a 
locally flat simple polyhedron $X$ with $L\subset \partial X$. 
Such a polyhedron is called a {\it shadow} of $(W,L)$.
To each internal region of $X$, we can assign a half integer from the embedding of $X$ in $W$ in a suitable way (see the next paragraph), called a {\it gleam}. 
Conversely, from a simple polyhedron $X$ assigned with gleams
we can recover a pair $(W,L)$ of a compact oriented $4$-manifold $W$ and a link $L\subset \partial W$ so that $X$ is a shadow of $(W,L)$ and the assigned half integers are the gleams. 
This method is called {\it Turaev's reconstruction}, see~\cite{Tur92, Tur94, Cos05}. 
We call the assignment of a half-integer to each internal region of $X$ as above, that is, a function from the set of internal regions of $X$ to $\frac{1}{2}\Integer$, a {\it gleam function} and denote it by $\gl$.
A simple polyhedron $X$ with a gleam function $\gl$ is called a {\it shadowed polyhedron} and denoted by $(X,\gl)$.
The half integer assigned to each region $R$ of $X$ by $\gl$ is called a gleam on $R$ and denoted by $\gl(R)$.

We explain how to 
assign
the gleam 
to an internal region of a shadow in 
an oriented 4-manifold.
Set $W_S=\Nbd(\Sing(X);W)$ and $X_S=X\cap W_S$. 
Since $X$ is locally flat, there exists a possibly non-orientable 3-dimensional 
handlebody $H_S$ in $W_S$ 
such that $\partial X_S = X_S \cap \partial H_S$, $\partial H_S = H_S \cap \partial W_S$ and $H_S$ collapses onto $X_S$.
See \cite{Mar05} for details. 
Let $R$ be an internal region. We define the {\it reference framing} of $R$ as 
$B_R=\Nbd(\partial(R\setminus \Int X_S) ; \partial H_S)$. 
Note that $B_R$ is a union of some annuli or M\"obius bands embedded in 
the solid tori $\Nbd(\partial X_S; \partial W_S)$
and also that the number of the M\"obius bands is determined only by the topological type of $X$. 
This number modulo $2$ is called the {\it $\Integer_2$-gleam} of $R$. 
Regard the (abstract) oriented $[-1,1]$-bundle over $(R\setminus \Int X_S) \times [0,1]$
as the disk bundle over $R\setminus \Int X_S$ and fix
an orientation preserving bundle isomorphism $\phi$ from this bundle to the normal bundle of 
$R\setminus \Int X_S$ in $W$.
The map $\phi$ sends the image of the zero section of the $[-1,1]$-bundle over $\partial(R\setminus \Int X_S) \times [0,1])$ to a union of some annuli in $\partial W_S$. We denote the union of these annuli by $B_R'$. 
The gleam $\gl(R)$ is then given as the number of times that 
$B_R'$ rotates with respect to $B_R$ in the solid tori $\Nbd(\partial X_S;\partial W_S)$.
More precisely, $\Nbd(\partial X_S;\partial W_S)$ inherits the orientation of $W$, and so the tori $\partial \Nbd(\partial X_S;\partial W_S)$ as well. 
The gleam $\gl(R)$ is defined to be the quarter of the intersection number of $\partial B_R'$ and $\partial B_R$ in $\partial \Nbd(\partial X_S; \partial W_S)$, where each components of $\partial B_R$ and $\partial B_R'$ are consistently oriented according to an arbitrarily fixed orientation of the circles $\partial X_S$.
Note that $\partial B_R'\cap \partial \Nbd(X_S; W_S)$ consists of two simple closed curves, while $\partial B_R\cap \partial \Nbd(X_S; W_S)$ consists of one or two simple closed curves.
The gleam $\gl(R)$ is an integer if and only if the $\Integer_2$-gleam of $R$ is $0$. 

\subsection{Immersed curve presentations}

\begin{definition}
An immersed curve $C$ on the unit disk $D$ is called an  {\it immersed curve presentation}
of a simple polyhedron $X$ if there exists a disk $\hat D$ in $X$ such that 
\begin{itemize}
\item $X\setminus \hat D$ is a disjoint union of copies of $S^1\times (0,1]$ whose closures do not intersect the boundary $\partial \hat D$ of $\hat D$ and
\item the pair $(D,C)$ is homeomorphic to the pair $(\hat D,\Sing(X))$. 
\end{itemize}
\end{definition}
 
An example of an immersed curve presentation of a simple polyhedron is given in Figure~\ref{fig1}. The polyhedron given by this immersed curve presentation is obtained from the disk $D$ by attaching an annulus along the immersed curve.

\begin{figure}[htbp]
\begin{center}
$ $ \phantom{aaaaaaaaa}\includegraphics[width=6.5cm, bb=160 531 407 712]{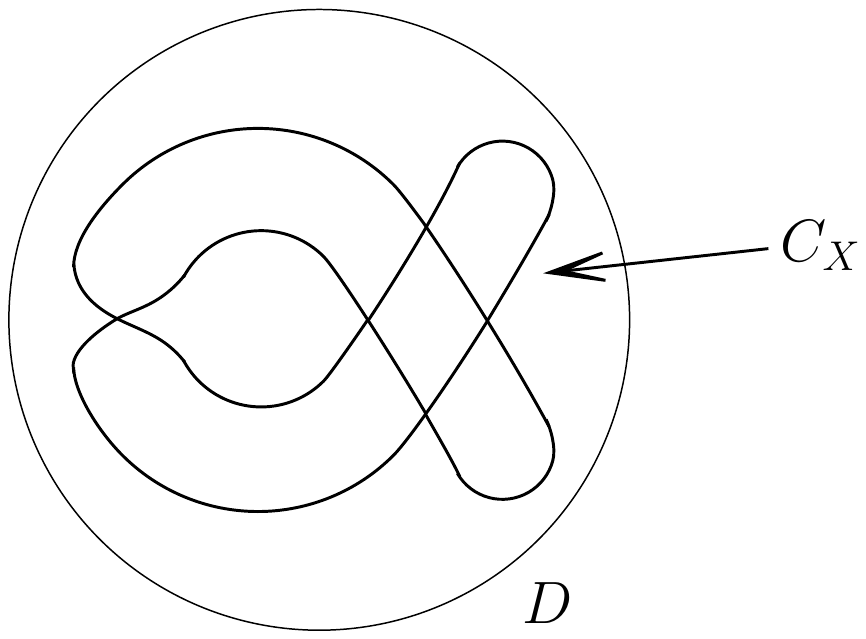}
\caption{An immersed curve presentation $C_X$.}\label{fig1}
\end{center}
\end{figure}

Let $C_X$ be an immersed curve presentation of a simple polyhedron $X$.
Hereafter we always assume that $C_X$ is {\it connected}, that is, $\Sing(X)$ is connected, for simplicity.
Note that the internal regions of $X$ are the regions of $X$ on $D$ that do not intersect the boundary $\partial D$ of $D$. 

\begin{definition}
A disjoint union $A$ of trees embedded in $D$ is called 
a {\it system of cutting trees} of $C_X$ if it satisfies that
\begin{itemize}
\item $C_X \setminus A$ is simply-connected, 
\item exactly one of the endpoints of each tree is on $\partial D$ and the others are on the internal regions of $X$, 
\item each internal region of $X$ contains exactly one vertex of $A$,
\item each edge of $A$ intersects $C_X$ transversely at one point, and
\item $C_X\cap A$ is away from the double points of $C_X$.
\end{itemize}
An endpoint of a tree lying in the interior of a region is called a {\it terminal point}.
\end{definition}

\begin{remark}
It is possible to generalize the results in this paper to the case where $\Sing(X)$ is not connected, though  the algorithm for a presentation of the fundamental group becomes complicated slightly in that case.
\end{remark}

A system of cutting trees of the immersed curve presentation in Figure~\ref{fig1} is given in Figure~\ref{fig2a}.

\begin{figure}[htbp]
\begin{center}
$ $ \phantom{aaaaaaaaa}\includegraphics[width=7cm, bb=148 531 419 712]{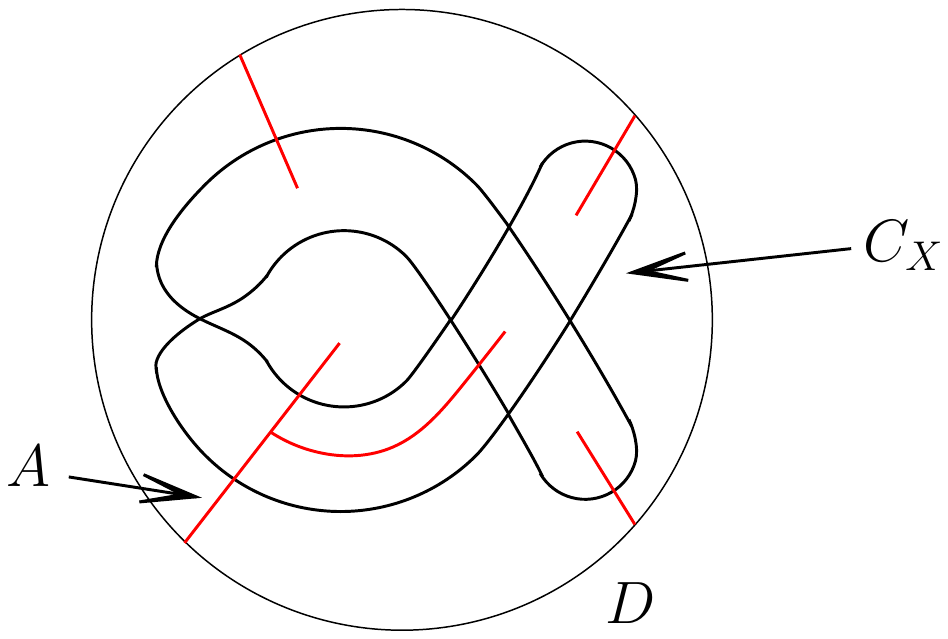}
\caption{A system of cutting trees.}\label{fig2a}
\end{center}
\end{figure}

\subsection{Link diagram presentations}\label{sec22}

Next we add over/under information to each double point of $C_X$ arbitrarily. 
We call the obtained diagram a {\it link diagram presentation} of $X$ and denote it by $D_X$. 
A union of trees on $D$ is called a {\it system of cutting trees} of $D_X$ if it is a system of cutting trees of $C_X$. 

The diagram $D_X$ consists of a finite number of arcs described on $D$. In this paper, an arc of a diagram is called a {\it strand}. A {\it crossing point} of $D_X$ means the point on $D_X$ corresponding to a double point of $C_X$. The intersection of $D_X$ and
a neighborhood of a crossing point of $D_X$ consists of three subarcs of strands. 
The subarc intersecting the crossing point is called the {\it overstrand} and the other two subarcs
are called the {\it understrands} of the crossing point.

Let $D_X$ be a link diagram presentation of $X$. For each internal region $R$ of $X$, let $c(R)$ be the sum of local contributions at the vertices of $X$ on the boundary of $R$ given as in Figure~\ref{fig4},
where the over/under information is that of $D_X$. This rule is the same as the one for shadow projections of links in~\cite{Tur94}.

\begin{figure}[htbp]
\begin{center}
\includegraphics[width=25mm, bb=259 637 334 713]{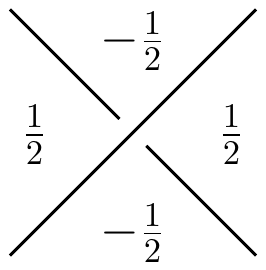}
\caption{Local contribution to $c(R)$.}\label{fig4}
\end{center}
\end{figure}

\subsection{Meridians}\label{sec24}

Let $(X,\gl)$ be a shadowed polyhedron with an immersed curve presentation. 
The $4$-manifold of $(X,\gl)$ is always a $4$-ball, which we denote by $B^4$.
Let $R_1, R_2, \ldots, R_{n_0}$ be the regions of $X$ on $D$,
and set $\mathcal R_i= R_i \setminus \Nbd(C_X;D)$ for $i=1,\ldots, n_0$. 
The $4$-ball $B^4$ of $(X,\gl)$ is obtained from $\Nbd(C_X;D)\times D'$ by attaching 
$\overline{\mathcal R_i}\times D'$ for $i=1,\ldots,n_0$ using the gluing maps determined by $\gl$, where $D'$ is the unit disk on $\Real^2$
and $\overline{\mathcal R_i}$ is the closure of $\mathcal R_i$ in $D$. 
The regions of $X$ not lying on $D$ are a finite copies of $S^1\times (0,1]$ attached to $D$ along $C_X$.

By a {\it meridian} of a region $R$ of a shadow $X$ of a $4$-manifold $W$, we mean a based closed path in $W\setminus X$, with base point $b\in W\setminus X$, 
that is freely homotopic to a simple loop bounding a disk that intersects $X$ once at 
a point in the interior of $R$ transversely. Our main theorem says that for a subpolyhedron $Y$ of $X$, $\pi_1(W\setminus Y, b)$ has a finite presentation whose generator set consists of meridians of regions.

To define the meridians of our presentation, 
we first fix the position of the part of $X$ in $\Nbd(C_X;D)\times D'$ explicitly. 
Let $U_\varepsilon$ be the union of small disks 
on $\Nbd(C_X;D)$ centered at the double points of $C_X$ with sufficiently small radius $\varepsilon>0$.
Let $D_X$ be a link diagram presentation of $X$
and we regard it as a diagram on $\Nbd(C_X;D)$ by restriction. 
Then we fix the positions of $S^1\times (0,1]$
in $\Nbd(C_X;D)\times D'$ 
so that
\begin{itemize}
\item[(i)] 
the part of $S^1\times (0,1]$ outside $U_\varepsilon\times D'$ 
lies on $C_X\times \{(u_1',0)\in D'\mid 0<u_1'\leq 1\}$,
\item[(ii)] the part of  $S^1\times (0,1]$ in $U_\varepsilon\times D'$ 
corresponding to the overstrands of $D_X$ is $\bar a  \times \{(u_1',0)\in D'\mid 0<u_1'\leq 1\}$, 
where $\bar a$ is the set of arcs on  $C_X\cap U_\varepsilon$ corresponding to the 
overstrands of $D_X$,
and
\item[(iii)]  the part of  $S^1\times (0,1]$ in $U_\varepsilon\times D'$ corresponding to the 
understrands 
of crossing points
of $D_X$ is 
\[
\begin{split}
\{(u_1, u_2, u_1', u_2') &\in \Nbd(C_X;D)\times D'\mid \\
&(u_1,u_2)\in {\underline{a}}\cap U_\varepsilon,\; u'_1+\sqrt{-1}u'_2=te^{-\pi\sqrt{-1}\chi(r)},\;  0<t\leq 1\},
\end{split}
\]
where $\underline{a}$ is the set of arcs on  $C_X\cap U_\varepsilon$ corresponding to the 
understrands of $D_X$, 
$(r,\theta)$ are the polar coordinates on each disk of $U_\varepsilon$, and $\chi:[0,\varepsilon]\to [0,1]$ is a smooth bump function that is  $1$ near $r=0$ and $0$ near $r=\varepsilon$.
\end{itemize}

The meridians of regions of $X$ on $D$ are defined as follows. 
Let $A$ be a system of cutting trees of $C_X$ and set $T=C_X\setminus A$,
where $T$ is a tree with open endpoints whose vertices are the vertices of $X$.
Fix a base point $b$ on 
$(T\setminus U_\varepsilon) \times \{(0,1)\}\subset \Nbd(C_X;D)\times D'$.
Assume that $\Nbd(C_X;D)$ is sufficiently narrow so that the terminal points of the cutting trees are on $D\setminus \Nbd(C_X;D)$.
The union $A$ of cutting trees may decompose $R_i$ and $\mathcal R_i$ into open disks, which we denote by $R_{i1}, \ldots, R_{in_i}$ and $\mathcal R_{i1}, \ldots, \mathcal R_{in_i}$, respectively.
Choose a point $p_{ij}$ on $R_{ij}\setminus \mathcal R_{ij}$ and a point $q_{ij}$ on an edge of $T$ adjacent to $R_{ij}$.

The {\it meridian of $R_{ij}$} is defined to be the loop concatenating the minimal path $\omega_{ij}$ on $T\times \{(0,1)\}$ from the base point $b$ to $(q_{ij}, (0,1))$, the straight path $\omega'_{ij}$ from $(q_{ij}, (0,1))$ 
to $(p_{ij}, (0,1))$, the circle path on $\{p_{ij}\}\times\partial D'$ parametrized 
from $\theta=\frac{\pi}{2}$ to $\frac{5\pi}{2}$ as $(p_{ij}, e^{\sqrt{-1}\theta})$, 
the inverse path of $\omega'_{ij}$ and then the inverse path of $\omega_{ij}$.
Note that the homotopy class of the loop does not depend on the choice of the points $p_{ij}$ and $q_{ij}$. 

To define the meridians of the edges of $T$, we  need to assign orientations to the immersed curves of $C_X$. 
Let $\vec{C}_X$ be the union of immersed curves of $C_X$ oriented arbitrarily, which we call 
an {\it oriented immersed curve presentation} of $X$.  
For each edge $e_i$ of $T$, let $y_{l_i}$ and $y_{r_i}$ be the meridians of the regions on the left and right, respectively, with respect to $e_i$ equipped with the orientation consistent with that of $\vec{C}_X$.
The {\it meridian of $e_i$} is defined to be the loop shown on the left in Figure~\ref{fig6b}, which satisfies $x_i=y_{r_i}^{-1}y_{l_i}$. 
We write down these meridians for each edge of $T$ as shown on the right.

\begin{figure}[htbp]
\begin{center}
\includegraphics[width=12cm, bb=130 601 454 712]{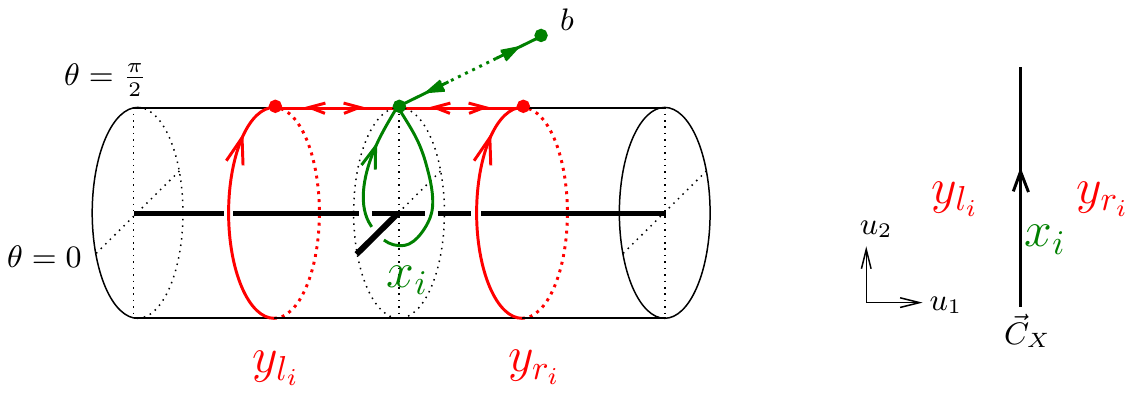}
\caption{Meridians of edges of $T$.}\label{fig6b}
\end{center}
\end{figure}

\begin{lemma}
Let $e$ and $e'$ be edges of $T$ adjacent to the same vertex of $T$.
Suppose that both $e$ and $e'$ correspond to the overstrand of $D_X$ at the crossing corresponding to the vertex. Then their meridians are homotopic
in $B^4\setminus X$.
\end{lemma}

\begin{proof}
It follows from the condition~(ii) of the positions of $S^1\times (0,1]$ corresponding to the overstrand
and the positions of the loops of the meridians.
\end{proof}

We can obtain an oriented link diagram from $D_X$ by assigning the orientation of $\vec{C}_X$ to the strands of $D_X$. To simplify the notation, we denote this oriented link diagram again by $D_X$. Due to the above lemma, the meridian of a strand of $D_X\setminus A$ is well-defined even if it passes through crossings as overstrands.

\section{Presentations of fundamental groups}

\subsection{Main result} 

To state our main theorem, we introduce some terminologies. 
Let $c_j$ be an intersection point of $D_X$ and $A$, called a {\it cutting point}, and let $A_j$ be the tree in $A$ containing $c_j$.
The point $c_j$ decomposes $A_j$ into two subtrees of $A_j$, and we denote the one not containing the point on $\partial D$ by $A_j'$. Let $R_j$ be the region of $X$ on $D$ that is adjacent to $c_j$ and intersects $A_j'$.
The point $c_j$ cuts the strand of $D_X$ containing $c_j$ into two arcs, which are stands of $D_X\setminus A$,  and we order them according to the counterclockwise orientation on the boundary of $R_j$. We call the first strand the {\it backward strand at $c_j$} and the second one the {\it forward strand at $c_j$}, 
see Figure~\ref{fig3}. We call $c_j$ the {\it cutting point of the region $R_j$}.

\begin{figure}[htbp]
\begin{center}
\includegraphics[width=9cm, bb=173 576 410 711]{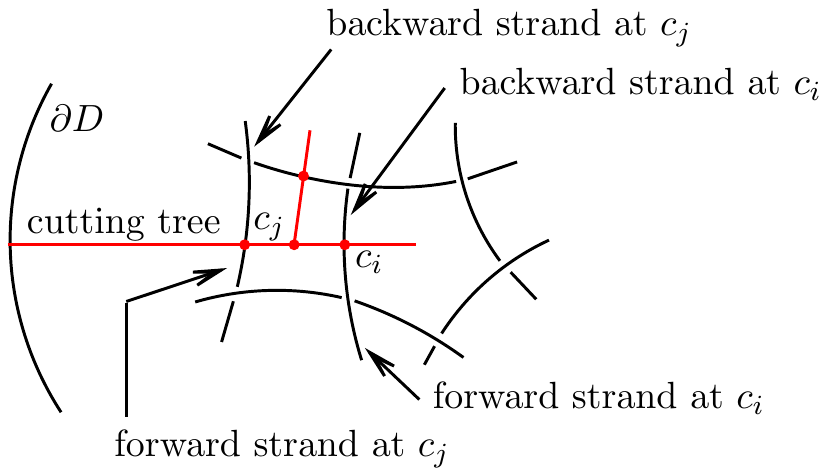}
\caption{Forward and backward strands at intersection points of $D_X$ and $A$.}\label{fig3}
\end{center}
\end{figure}

Let $Y$ be a simple subpolyhedron of $X$ satisfying $X\setminus D\subset Y$.
Here $Y$ is called a {\it subpolyhedron} of $X$ if $Y$ is obtained from $X$ by removing some regions and edges of $X$. 
Let $R_{j_1}, R_{j_2}, \ldots, R_{j_k}$ be the regions of $Y$ on $D$ intersecting $A_j'$. 
Let $\Nbd(A_j';D)$ be a small neighborhood of $A_j'$ in $D$. 
Its boundary $\partial \Nbd(A_j';D)$ is a simple closed curve that intersects all the regions $R_{j_1}, R_{j_2}, \ldots, R_{j_k}$. 
We give the suffices $j_1, j_2,\ldots, j_k$ according to the following rule. 
Take a point on  $\partial \Nbd(A_j';D)$ near $c_j$ and
travel counterclockwise on the curve $\partial \Nbd(A_j';D)$. 
If $1\leq k'<k''\leq k$, then the curve meets $R_{j_{k'}}$ for the first time before it meets $R_{j_{k''}}$ for the first time. See Figure~\ref{fig3-2}. 
We say 
the ordering of the regions of $Y$ intersecting $A'_j$ given by the above rule on the suffices
the {\it counterclockwise ordering}.

\begin{figure}[htbp]
\begin{center}
\includegraphics[width=7cm, bb=195 596 364 711]{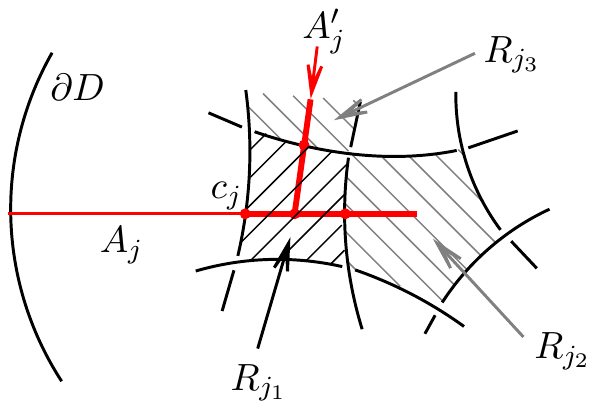}
\caption{The subtree $A_j'$ and the ordered regions $R_{j_1}, R_{j_2}, R_{j_3},\ldots$ The thickened tree in the figure is $A_j'$.}\label{fig3-2}
\end{center}
\end{figure}

Now we state the main theorem.

\begin{theorem}\label{thm1}
Let $(X,\gl)$ be a shadowed polyhedron with an oriented link diagram presentation $D_X$ on a disk $D$.
Let $Y$ be a simple subpolyhedron of $X$ satisfying $X\setminus D\subset Y$. 
Then, for a system $A$ of cutting trees of $D_X$, we have 
\[ 
\pi_1(B^4\setminus Y)\cong
\langle
x_1, \ldots, x_n, y_1,\ldots,y_m \mid
s_1,\ldots, s_{n'},
t_1,\ldots, t_{n''}
\rangle,
\]
where $x_1,\ldots, x_n$ are the meridians of the strands of $D_X\setminus A$, $y_1,\ldots, y_m$ are the meridians of the regions of $Y\setminus A$ on $D$,  
$s_i=y_{r_i}x_iy_{l_i}^{-1}$ is the 
relator obtained for each edge $e_i$ of $T$, where $y_{l_i}=1$ {\rm (}resp. $y_{r_i}=1$\,{\rm )} if the region on the left {\rm (}resp. right\,{\rm )} of $e_i$ is not contained in $Y$,
and 
$t_j=\gamma_j x_{f_j}\gamma_j^{-1}x_{b_j}^{-1}$ is the 
relator 
obtained for each cutting point $c_j$, where
$x_{f_j}$ is the meridian of the forward strand and $x_{b_j}$ is that of the backward strand at $c_j$ and
\begin{equation}\label{eq103}
\gamma_j=y_{\varphi(j_k)}^{\gl(R_{j_k})-c(R_{j_k})}\cdots y_{\varphi(j_2)}^{\gl(R_{j_2})-c(R_{j_2})}y_{\varphi(j_1)}^{\gl(R_{j_1})-c(R_{j_1})},
\end{equation}
where
\begin{itemize}
\item $c(R)$ is the sum of local contributions to a region $R$ introduced in Section~\ref{sec22},
\item $R_{j_1}, R_{j_2}, \ldots, R_{j_k}$ are the regions of $Y$ on $D$ intersecting the subtree $A_j'$ aligned in  counterclockwise ordering, and 
\item 
$y_{\varphi(j_1)}, y_{\varphi(j_2)}, \ldots, y_{\varphi(j_k)}$ 
are the meridians of the regions of $Y\setminus A$ contained in $R_{j_1}, R_{j_2}, \ldots, R_{j_k}$ and adjacent to the forward strands at the cutting points of $R_{j_1}, R_{j_2}, \ldots, R_{j_k}$, respectively. Here
$\varphi$ is the map that sends the suffix of $R_j$ to the suffix of the region of $Y\setminus A$ contained in $R_j$ and adjacent to the forward strand at the cutting point of $R_j$.
\end{itemize}
\end{theorem}

\begin{proof}
We prove the assertion in the case $Y=X$. The assertion in the other cases can be proved by setting the meridians of the regions of $X$ not contained in $Y$ to be the identity.

Set $N' = \mathrm{Nbd} (C_X; D) \setminus \mathrm{Int} \, \mathrm{Nbd} (A; D )$ and $B' = N'  \times D'$. 
We first show that $\pi_1 (B' \setminus X)$ has the presentation $\langle x_1, \ldots, x_n, y_1, \ldots , y_m \mid  s_1 , \ldots, s_{n'} \rangle$. 
By the definition of a system of cutting trees, $N'$ is a closed disk, thus, $B'$ is a $4$-ball. 
We decompose $N'$ into the pieces $V_1, \ldots, V_{n_v}, E_1, \ldots, E_{n'}$, where 
$V_1, \ldots, V_{n_v}$ and $E_1, \ldots, E_{n'}$ correspond to the vertices and edges of $T=C_X\setminus A$,
respectively. 
See Figure~\ref{fig_added_2}.

\begin{figure}[htbp]
\begin{center}
\includegraphics[width=10cm, bb=186 639 395 712]{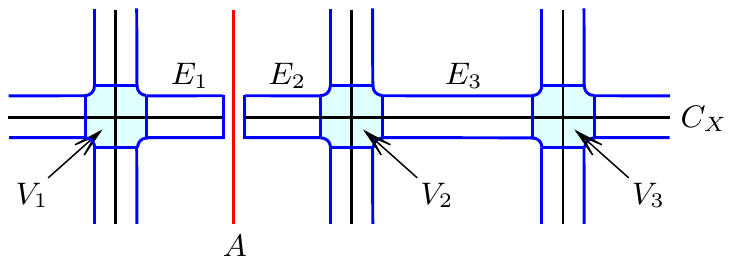}
\caption{The decomposition of $N'$ into the pieces $V_1,\ldots, V_{n_v}$ and $E_1,\ldots, E_{n'}$.}\label{fig_added_2}
\end{center}
\end{figure}

For each $1 \leq i \leq n'$, the pair $(E_i  \times D', (E_i  \times D') \cap X )$ is homeomorphic to 
the product space of the cone on $(S^2, P)$ and the interval $[0,1]$,
where $P \subset S^2$ is a 3-point set. 
Thus, the fundamental groups of $(E_i  \times D') \setminus X$ is a free group of rank $2$. 
Around the edge $e_i$ of $T$ corresponding to the piece $E_i$, there are three meridians $x_i, y_{r_i}, y_{l_i}$. 
As we have already seen in Figure~\ref{fig6b},
they satisfy the relation $x_i = y_{r_i}^{-1} y_{l_i}$, and thus, 
the product space $(E_i  \times D' ) \setminus  X $ has the following presentation: 
\[
 \langle  x_i, y_{r_i}, y_{l_i} 
\mid 
y_{r_i} x_i y_{l_i}^{-1}
\rangle,
\]
which is actually the free group of rank $2$. 

For each $1 \leq k \leq n_v$, the pair $(V_k  \times D', (V_k  \times D') \cap X )$ is homeomorphic to the cone on 
$(S^3, \Gamma)$, where $\Gamma$ is a $3$-regular graph with $4$ vertices planarly embedded in $S^3$. 
Therefore, the fundamental groups of both $(V_k  \times D') \setminus X$ and $\partial (V_k  \times D') \setminus X$ are 
the free group of rank $3$ and they can be naturally identified. 
Around the crossing point of concern, there are seven meridians. 
Suppose that the crossing of $D_X$ at the vertex is positive and we label the meridians as in 
Figure~\ref{fig25},
where
$\bar{x}$ is the meridian of the overstrand, $\underline{x}_1$ and $\underline{x}_2$ are the meridians of the understrands of the 
crossing, and $y_{i,l}$ and $y_{i,r}$ are the meridians of the regions on the left and right of the strand with the meridian 
$\underline{x}_i$ for each $i=1,2$, respectively. 
Precisely speaking, the base points
of the meridians here are different from those in the statement of the theorem, but 
we do not go into details on this difference for simplicity of exposition. 
As we have already seen in Figure~\ref{fig6b},
these meridians satisfy the relations 
$y_{1,l} \bar{x} y_{2,l}^{-1} = 1$, $y_{1,r} \bar{x} y_{2,r}^{-1} = 1$ and 
$y_{i,r} \underline{x}_i y_{i,l}^{-1} = 1$ for $i=1,2$. 
It is then easily checked that the fundamental group of $(V_k \times D')  \setminus X$ has 
the following presentation:
\[ \langle  \bar{x}, \underline{x}_1, \underline{x}_2, y_{1,l} , y_{1,r}, y_{2,l} , y_{2,r} 
\mid 
y_{1,l} \bar{x} y_{2,l}^{-1}, \,
y_{1,r} \bar{x} y_{2,r}^{-1}, \, 
y_{1,r} \underline{x}_1 y_{1,l}^{-1} , \,
y_{2,r} \underline{x}_2 y_{2,l}^{-1} , \,
y_{1,l} y_{1,r}^{-1} y_{2,r} y_{2,l}^{-1} 
\rangle ,
\]
which is actually the free group of rank $3$. 
We note that, as shown in Figure~\ref{fig26}, we have $y_{1,l} y_{1,r}^{-1} = y_{2,l} y_{2,r}^{-1}$ 
due to the condition (iii) of the positions of 
$S^1 \times (0,1]$ on $U_{\varepsilon}$. 
This relation can be derived from $y_{1,l} \bar{x} y_{2,l}^{-1} = 1$ and $y_{1,r} \bar{x} y_{2,r}^{-1} = 1$. 
The argument for the case of negative crossing runs in the same way. 

\begin{figure}[htbp]
\begin{center}
\includegraphics[width=3cm, bb=200 614 300 710]{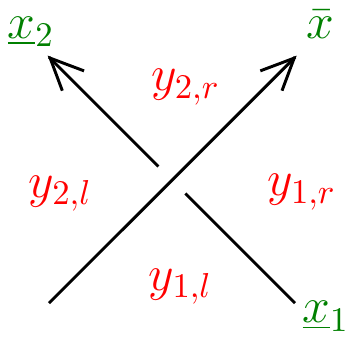}
\caption{Meridians around a vertex.}\label{fig25}
\end{center}
\end{figure}

\begin{figure}[htbp]
\begin{center}
\includegraphics[width=12.5cm, bb=164 373 404 712]{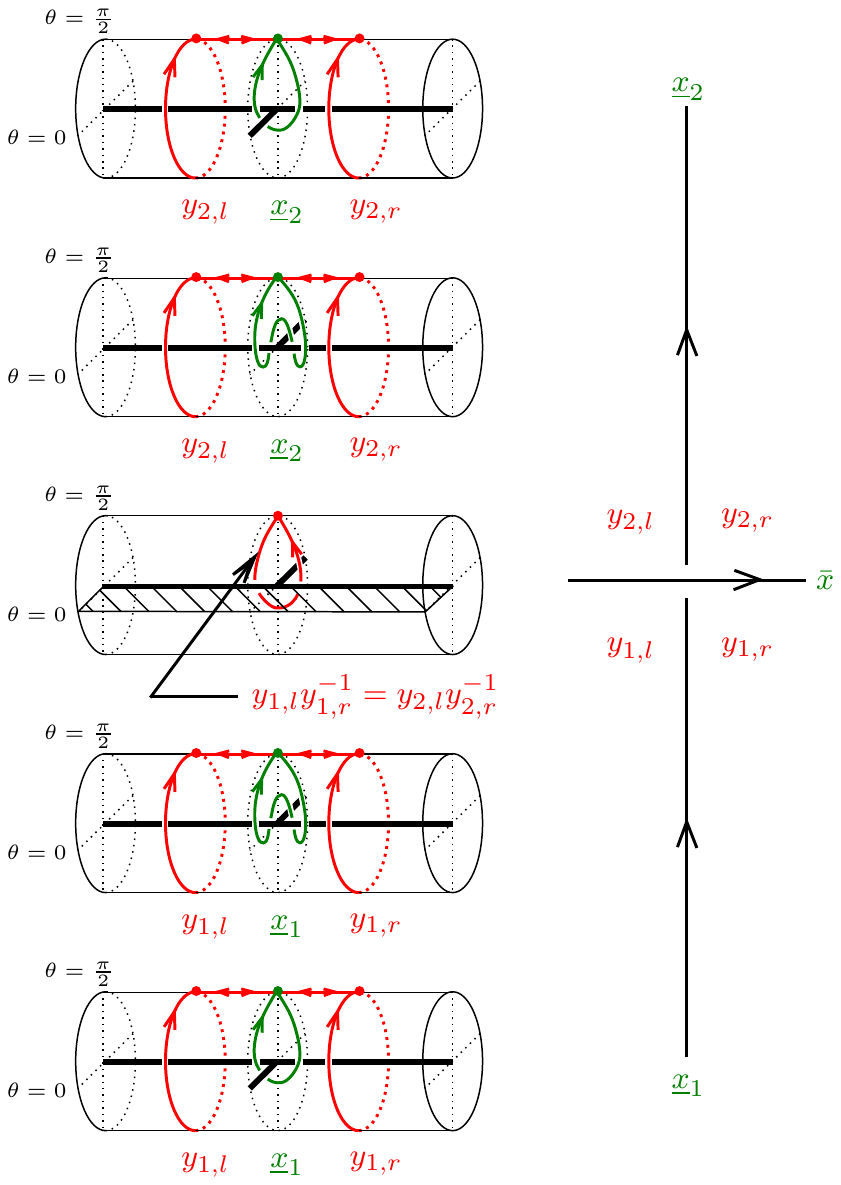}
\caption{Meridians around a vertex.}\label{fig26}
\end{center}
\end{figure}

Now we are ready to give a presentation for $\pi_1 (B'  \setminus X)$. 
Suppose that  $V_k  \cap E_{k'} = \partial V_k  \cap \partial E_{k'}   \neq \emptyset$. 
Note that by the assumption that $T=C_X\setminus A$ is a tree, $\partial V_k\cap \partial E_{k'}$ is connected. 
Then, by the construction the pair $((V_k   \cap E_{k'}) \times D', ((V_k \cap E_{k'}) \times D') \cap X)$ 
is homeomorphic to the cone on $(S^2, P)$, where $P \subset S^2$ is as above. 
Thus, the fundamental group of  $((V_k \cap E_{k'}) \times D') \setminus X$ is 
a free group of rank $2$. 
Further, the maps $\pi_1 ( ( ( \partial V_k  \cap \partial E_{k'}) \times D' ) \setminus X) \to 
\pi_1 ((V_k \times D') \setminus X)$ and 
$\pi_1 ( (\partial V_k  \cap \partial E_{k'})\times D') \setminus X ) \to \pi_1 ((E_{k'} \times D') \setminus X)$ 
induced from the inclusion maps are monomorphisms. 
Therefore, by applying van Kampen's theorem finitely many times with 
checking the images of the above monomorphisms, we have 
\begin{equation}\label{eq101}
\langle x_1, \ldots, x_n, y_1, \ldots , y_m \mid  s_1 , \ldots, s_{n'} \rangle
\end{equation}
as a presentation of $\pi_1 (B' \setminus X)$.

Next we set $B''=\Nbd(C_X;D)\times D'$ and observe $\pi_1(B''\setminus X)$.
The manifold $B''\setminus X$ is homeomorphic to the union of $B'\setminus X$ and 
$((\Nbd(C_X;D)\cap \Nbd(A;D))\times D' )\setminus X$. 
For each cutting point $c_j$, there is a unique minimal closed curve on 
$(T\cup\{c_j\})\times \{(0,1)\}$ based at $b$, possibly with self-intersection, that passes through the point $(c_j, (0,1))$ exactly once.
We orient this loop so that the simple closed curve on this loop is oriented counterclockwise on 
$\Nbd(C_X;D)\times\{(0,1)\}$. 
Here the orientation on $\Nbd(C_X;D)\times\{(0,1)\}$ is chosen 
so that it coincides with that on 
$\Nbd(C_X;D)$ via the projection $\Nbd(C_X;D)\times\{(0,1)\}\to \Nbd(C_X;D)$. 
We denote this oriented loop by $\gamma_j$.
Then we see that $\pi_1(B''\setminus X)$ is obtained from the presentation of $\pi_1(B'\setminus X)$ in~\eqref{eq101} by adding, for each cutting point $c_j$, the generator $\gamma_j$ and the relations
\begin{equation}\label{eq104}
x_{b_j}=\gamma_jx_{f_j}\gamma_j^{-1}, 
\quad
y_{b_{j,l}}=\gamma_jy_{f_{j,l}}\gamma_j^{-1}
\quad\text{and}
\quad
y_{b_{j,r}}=\gamma_jy_{f_{j,r}}\gamma_j^{-1}, 
\end{equation}
where $x_{f_j}$ is the meridian of the forward strand and $x_{b_j}$ is that of the backward strand at $c_j$, 
$y_{f_{j,l}}$ and $y_{b_{j,l}}$ are the meridians of the regions on the left of $x_{f_j}$ and  $x_{b_j}$, respectively, 
and $y_{f_{j,r}}$ and $y_{b_{j,r}}$ are those on the right.
Remark that we can obtain one of them from the other two by using the relations
$y_{f_{j,r}}x_{f_j}y_{f_{j,l}}^{-1}=1$ and $y_{b_{j,r}}x_{b_j}y_{b_{j,l}}^{-1}=1$.

Finally, we consider $\pi_1(B^4\setminus X)$. 
For each region 
$R_j$ of $X$ on $D$ not containing a terminal point of $A$, 
let $R_j'$ 
be the region of $X\setminus A$ contained in 
$R_j$ 
and adjacent to the forward strand at the cutting point 
$c_j$ of $R_j$ 
see Figure~\ref{fig27}. 
If a region 
$R_j$ 
of $X$ on $D$ contains a terminal point of $A$, then we set 
$R'_j=R_j$.
Note that 
$y_{\varphi(j_1)},\ldots, y_{\varphi(j_k)}$ in the definition of $\gamma_j$ in the assertion are the meridians of $R'_{j_1},\ldots, R'_{j_k}$, respectively. 

\begin{figure}[htbp]
\begin{center}
\includegraphics[width=14cm, bb=129 591 488 711]{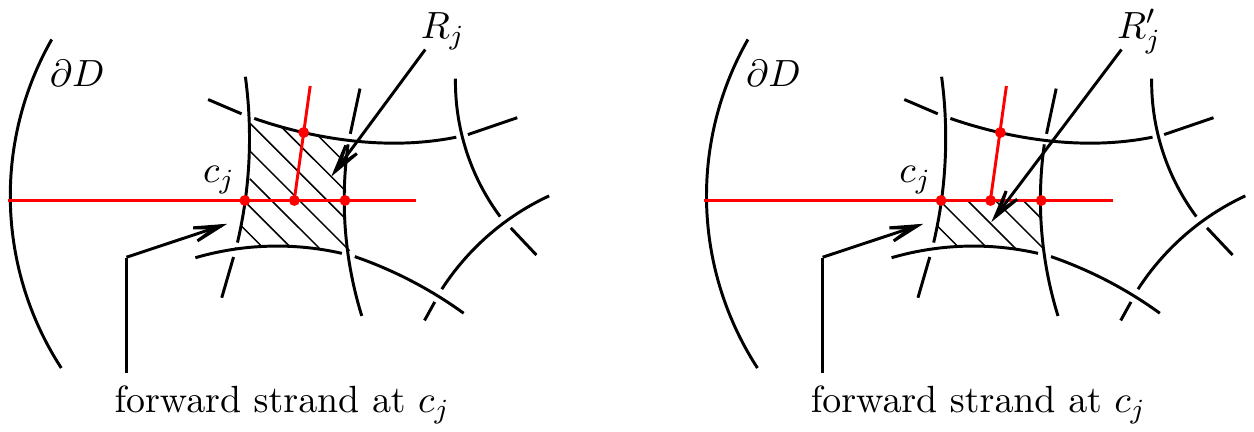}
\caption{The regions $R_j$ and $R_j'$.}\label{fig27}
\end{center}
\end{figure}

The manifold $B''\setminus X$ is obtained from $B^4\setminus X$ by removing  
$\mathcal R_i\times (D'\setminus \{(0,0)\})$ for each region $R_i$ of $X$ on $D$. 
If $R_i$ is the region containing $\partial D$ then the removal of $\mathcal R_i\times (D'\setminus \{(0,0)\})$ 
is realized by a deformation retract. Therefore, it does not change the fundamental group.
Instead of removing $\mathcal R_i\times (D'\setminus \{(0,0)\})$ 
for each internal region $R_i$ of $X$, we remove $(\mathcal R_i\cap R_i') \times (D'\setminus \{(0,0)\})$. 
Set $S=D\setminus\bigcup_{i=1}^{n_0}(\mathcal{R}_i \cap R'_i)$, which is a regular neighborhood of $C_X$ in $D$. 
The $4$-ball $B^4$ of $(X,\gl)$ is obtained from $S\times D'$ by attaching $R_i'\times D'$ for $i=1,\ldots,n_0$ using the gluing maps determined by $\gl$.
For each cutting point $c_i$, let $\delta_i$ be the loop obtained by concatenating the minimal path $\omega_i$ on $T\times \{(0,1)\}$ from the base point $b$ to $(q_i, (0,1))$, 
where $q_i$ is a point on the forward strand at $c_i$, 
a straight path $\omega'_i$ from $(q_i, (0,1))$ to a point $(p_i, (0,1))$ on $\partial (\mathcal R_i\cap R_i') \times \{(0,1)\}$, the circle path on  $\partial (\mathcal R_i\cap R_i')\times\{(0,1)\}$ parametrized counterclockwise, the inverse path of $\omega_i'$ and then the inverse path of $\omega_i$. 
Remark that we have the relation 
\begin{equation}\label{eq102}
   \gamma_j=\delta_{j_k}\delta_{j_{k-1}}\cdots\delta_{j_2}\delta_{j_1},
\end{equation}
where the
suffices $j_1, j_2, \ldots, j_k$ are those of the regions $R_{j_1}, R_{j_2}, \ldots, R_{j_k}$ of 
$X$ on $D$ intersecting the subtree $A_j'$ aligned in  counterclockwise ordering.

The reference framing of each internal region $R_i$ is given as follows. 
Fix a strong deformation retract $\{\phi_t:S\to C_X\mid t\in[0,1]\}$. 
For $u\in C_X$, let $\theta(u)$ be the argument of 
the segment $(X \setminus D)\cap (\{u\}\times D')$ on $\{u\}\times \Real^2$. 
This is well-defined modulo $\pi$
at the vertices of $X$ since the arguments of the regions corresponding 
to the overstands and understrands are $0$ and $\pi$, respectively.
Set
\[
H_S=
\left\{
(u,(r, \theta)) \in S \times D' 
\mid
0 \leq r \leq 1, 
\theta = \theta(\phi_1(u)) \text{ or } \theta(\phi_1(u)) + \pi 
\right\},
\]
where $(r,\theta)$ are the polar coordinates on $D'$.
The reference framing of $R_i$ is given as the annulus or M\"obius band 
$\Nbd(\partial(\mathcal{R}_i \cap R'_i);\partial H_S)$.

To observe the influence of the gleam on the presentation of $\pi_1(B^4 \setminus X)$, 
we use the $3$-manifold $\partial S\times D'$ 
in accordance with the definition of the gleam. 
Note that, in this $3$-manifold $\partial S\times D'$, 
the oriented loop $\delta_i$ is negatively transverse to 
$\{u\}\times D'$ for $u\in \partial (\mathcal{R}_i \cap R'_i)$ 
since the orientation of $\partial S\times D'$ is induced from $S\times D'$ 
but not from $(\mathcal{R}_i \cap R'_i)\times D'$.
For example, the arc on $\partial (\mathcal R_i\cap R_i')$ for the region $R_i$ shown on the right in Figure~\ref{fig_added} is oriented from $u_1$ to $u_3$, which is opposite to the orientation of $\delta_i$. 
This orientation and the orientation of the disk $D'$ with coordinates $(r,\theta)$ give the orientation of the $3$-manifold $\partial S\times D'$. The reference framing $\Nbd(\partial (\mathcal R_i\cap R_i'); \partial H_S)$ is given by the band shown on the left in the figure. This band is not twisted from $u_1$ to $u_2$ since it corresponds to the overstrand, and it is twisted by $+\pi$ from $u_2$ to $u_3$ since it corresponds to the understrands and the part of $X$ rotates as explained in~(iii) in Section~\ref{sec24}. Since the local contribution of this corner to $c(R_i)$ is $-\frac{1}{2}$, we can conclude that the rotation of the reference framing is $-1$ times the local contribution. This observation is also true for the other three regions at the vertex.
Thus, for each region $R_i$, the loop $\partial(\mathcal{R}_i \cap R'_i)\times \{(0,1)\}$ 
rotates $-c(R_i)\times 2\pi$
with respect to the reference framing in $\partial S\times D'$.

\begin{figure}[htbp]
\begin{center}
\includegraphics[width=11.5cm, bb=184 642 400 709]{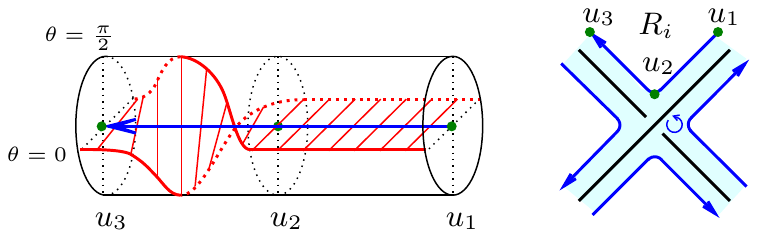}
\caption{The reference framing.}\label{fig_added}
\end{center}
\end{figure}

By the definition of the gleam, the $4$-dimensional block 
$(\mathcal{R}_i \cap R'_i)\times D'$ is glued to $S\times D'$ so that 
$\partial (\mathcal{R}_i \cap R'_i)\times \{(u'_1,0)\in D'\mid -1\leq u'_1\leq 1\}$ 
rotates 
$\gl(R_i)\times 2\pi$ 
with respect to the reference framing in 
$\partial S\times D'$. 
Hence the loop $y_i^{\gl(R_i)-c(R_i)}\delta_i^{-1}$ is nullhomotopic in $B^4\setminus X$ by van Kampen's theorem. 
Thus we have the relation 
\begin{equation}\label{eq1000}
\delta_{j_s}=y_{j_s}^{\gl(R_{j_s})-c(R_{j_s})}
\end{equation}
for $s=1,\ldots,k$. 
Substituting these relations into~\eqref{eq102}, we have the equality in~\eqref{eq103} in the assertion. 

To complete the proof, we need to show that the second and third relations in~\eqref{eq104} are not necessary.
For each cutting point $c_j$, let $y_{f_{j,1}}$ and $y_{b_{j,1}}$ be the meridians of the regions of $X\setminus A$ contained in $R_j$ and adjacent to the forward and backward strands at $c_j$, respectively, and
let $y_{f_{j,2}}$ and $y_{b_{j,2}}$ be the meridians of the regions of $X\setminus A$ not contained in $R_j$ and adjacent to the forward and backward strands at $c_j$, respectively.
With these notations, the relations in~\eqref{eq104} can be written as
\begin{equation}\label{eq106}
x_{b_j}=\gamma_jx_{f_j}\gamma_j^{-1}, 
\quad
y_{b_{j,1}}=\gamma_jy_{f_{j,1}}\gamma_j^{-1}
\quad\text{and}
\quad
y_{b_{j,2}}=\gamma_j y_{f_{j,2}}\gamma_j^{-1}.
\end{equation}

Now, we are going to show that for any cutting point $c_j$, 
the words $\gamma_j y_{f_{j,1}} \gamma_j^{-1} y_{b_{j,1}}^{-1}$ and $\gamma_j y_{f_{j,2}} \gamma_j^{-1}y_{b_{j,2}}^{-1}$ are consequences of the relators 
$s_1, \ldots, s_{n'}$ and $t_1, \ldots, t_{n''}$. 
It is easily checked that the third relation in~\eqref{eq106}
follows from the first and second relations. 
Thus, it suffices to show that 
$\gamma_j y_{f_{j,1}} \gamma_j^{-1}y_{b_{j,1}}^{-1}$ 
is a consequence of the relators 
$s_1, \ldots, s_{n'}$ and $t_1, \ldots, t_{n''}$. 
For this purpose, we define the \textit{size} of $c_j$ to be the number of cutting points contained in $A'_j$ except $c_j$. 
The proof is by induction on the size of $c_j$.

Suppose that the size of $c_j$ is zero, that is, $R_{j_s}$ contains a terminal point of $A'_j$. 
Then, we have $\gamma_{j_s} = \delta_{j_s}$ 
by~\eqref{eq102}.
We also have $y_{f_{j_s, 1}} = y_{b_{j_s, 1}}$ for they correspond to the same component $R_{j_s}$ of the regions 
of $X \setminus A$ on $D$. 
Thus, the word $\gamma_jy_{f_{j,1}}\gamma_j^{-1}y_{b_{j,1}}^{-1}$, which corresponds to the second relation 
in~\eqref{eq106}, is a consequence of the relators $s_1, \ldots, s_{n'}$ and $t_1, \ldots, t_{n''}$ by~\eqref{eq1000}.

For the inductive step, let $h > 0$ be an integer, and assume that the assertion is true for all cutting points of size less than $h$. 
Let $c_{j_s}$ be a cutting point of size $h$. 
On the boundary of the region $R_{j_s}$, there are at most $h$ cutting points $c_{k_1}, c_{k_2}, \ldots, c_{k_{h'}}$ except 
$c_{j_s}$, where we order them counterclockwise from $c_{j_s}$. 
Note that by definition we have $y_{f_{j_s, 1}} = y_{f_{k_1, 2}}$,  $y_{b_{k_t, 2}} = y_{f_{k_{t+1}, 2}}$ ($t = 1, \ldots, h'-1$), 
$y_{b_{k_{h'}, 2}} = y_{b_{j_s, 1}} $. 
See Figure~\ref{fig28}.
By the assumption of induction, 
the word $\gamma_{k_t} y_{f_{k_t, 1}} \gamma_{k_t}^{-1}y_{b_{k_t, 1}}^{-1}$ 
is a consequence of the relators $s_1, \ldots, s_{n'}$ and $t_1, \ldots, t_{n''}$ for 
$t = 1, \ldots, h'-1$. 
As we have explained above, the same thing holds for 
$\gamma_{k_t} y_{f_{k_t, 2}} \gamma_{k_t}^{-1}y_{b_{k_t, 2}}^{-1}$. 
Further, the counterclockwise ordering of $c_{k_1}, c_{k_2}, \ldots, c_{k_{h'}}$ implies 
 $\gamma_{j_s} = \gamma_{k_{h'}} \cdots \gamma_{k_2} \gamma_{k_1} \delta_{j_s}$. 
From these, we have 
\[ y_{b_{j_s, 1}} = y_{b_{k_{h'}, 2}} = \gamma_{k_{h'}} y_{f_{k_{h'}, 2}} \gamma_{k_{h'}}^{-1} 
= \gamma_{k_{h'}} \cdots \gamma_{k_1} y_{f_{k_1, 2}} \gamma_{k_1}^{-1}  \cdots \gamma_{k_{h'}}^{-1} 
= \gamma_{j_s} \delta_{j_s}^{-1}  y_{f_{j_s, 1}} \delta_{j_s} \gamma_{j_s}^{-1}.
\]  
 Since $ y_{f_{j_s}, 1}$ commutes with $\delta_{j_s}$, we have 
$y_{b_{j_s, 1}} = \gamma_{j_s} y_{f_{j_s, 1}} \gamma_{j_s}^{-1}$. 
This implies that 
the word $\gamma_{j_s} y_{f_{j_s, 1}} \gamma_{j_s}^{-1} y_{b_{j_s, 1}}^{-1}$ 
is a consequence of the relators $s_1, \ldots, s_{n'}$ and $t_1, \ldots, t_{n''}$. 
This completes the proof.
\end{proof}

\begin{figure}[htbp]
\begin{center}
\includegraphics[width=11cm, bb=138 546 406 712]{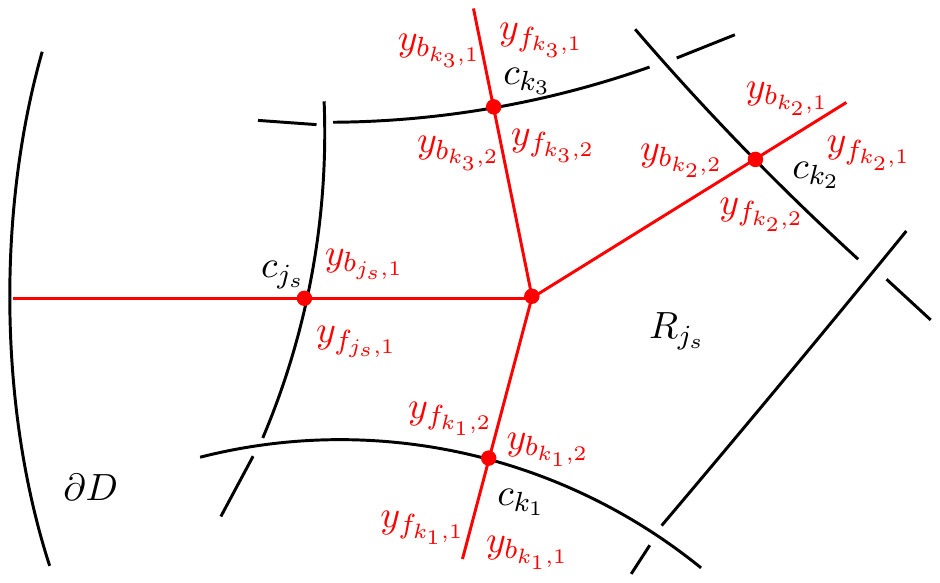}
\caption{Meridians $y_{f_{k_t,i}}$, $y_{b_{k_t,i}}$ for $i=1,2$ and $t=1,\ldots, h$.}\label{fig28}
\end{center}
\end{figure}

\begin{remark}
The condition $X\setminus D\subset Y$ for a subpolyhedron $Y$ is not essential.
If one wants to consider the case  $X\setminus D\not\subset Y$, one should use $Y\cup D$ instead of $X$.
\end{remark}

\begin{remark}
If we reverse the orientations of the meridians, the 
relator $y_{r_i}x_iy_{l_i}^{-1}$ changes to $y_{l_i}^{-1}x_iy_{r_i}$ 
since we obtain the relation $y_{r_i}^{-1}x_i^{-1}(y_{l_i}^{-1})^{-1}=1$ by the reversal. 
Thus, the reversal of the orientation of the meridians corresponds to the reversal of the order of the words of the relators. By the reversal, $\gamma_j$ should be replaced by 
\[
\gamma_j=
y_{\varphi(j_k)}^{c(R_{j_k})-\gl(R_{j_k})}
\cdots y_{\varphi(j_2)}^{c(R_{j_2})-\gl(R_{j_2})}y_{\varphi(j_1)}^{c(R_{j_1})-\gl(R_{j_1})},
\]
where the signs of the powers change since the orientation of $D'$ is reversed,  but the order of $j_1, j_2, \ldots, j_k$ 
does not change since the paths from $b$ to the regions do not change by the reversal.
\end{remark}

\subsection{Case with $Y\cap \partial D=\emptyset$}

In this section, we assume that a subpolyhedron $Y$ of a shadowed polyhedron $(X,\gl)$ with an immersed curve presentation satisfies $X\setminus D\subset Y$ and $Y\cap \partial D=\emptyset$. 
Let $C_X$ be the immersed curve presentation of $X$ on $D$.
As mentioned in the introduction, the presentation in Theorem~\ref{thm1} can be used for presenting the fundamental groups of the complements of Milnor fibers and those of complexified real line arrangements. 
These objects can be given by a shadowed polyhedron $(X,\gl)$ with a subpolyhedron $Y$ satisfying the above assumptions. 
When we calculate these fundamental groups, we can simplify the calculation slightly. 
We first introduce a reduced version of a system of cutting trees.

\begin{definition}
A disjoint union of trees on $D$ obtained from a system $A$ of cutting trees of $C_X$
by removing all cutting trees intersecting $C_X$ only once is called a {\it reduced system of cutting trees} of $C_X$. 
A disjoint union of trees on $D$ is said to be a {\it reduced system of cutting trees} of $D_X$ if it is a reduced system of cutting trees of $C_X$.
\end{definition}

A system of cutting trees of the immersed curve presentation in Figure~\ref{fig1} is given in Figure~\ref{fig2}.
The dotted arcs are the trees  intersecting $C_X$ only once, and the union of the solid trees is a reduced system of cutting trees.

\begin{figure}[htbp]
\begin{center}
$ $ \phantom{aaaaaaaaa}\includegraphics[width=7cm, bb=148 531 419 712]{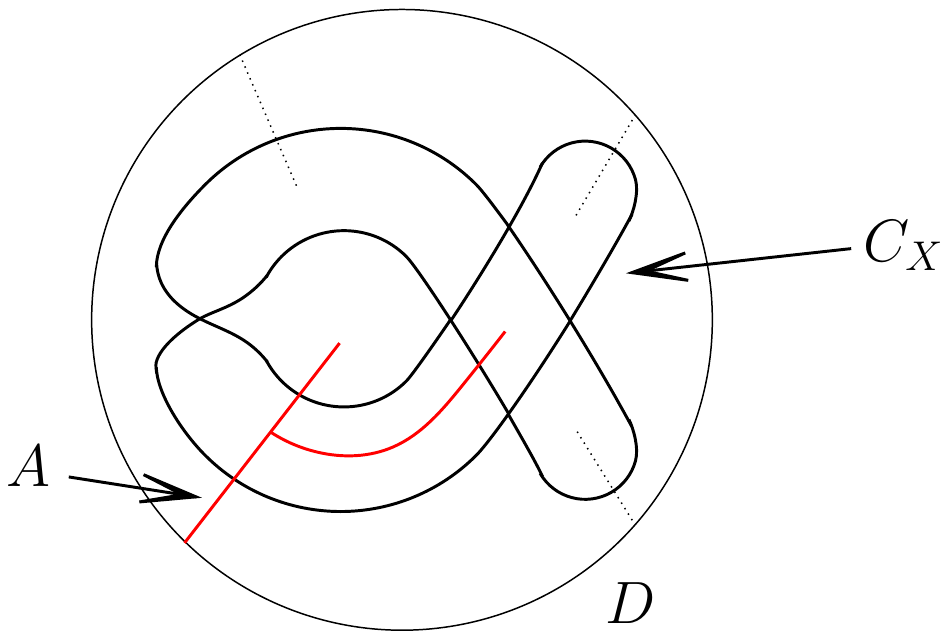}
\caption{A reduced system of cutting trees.}\label{fig2}
\end{center}
\end{figure}

Let $A^{\mathrm{red}}$ be the reduced system obtained from a system $A$ of cutting trees of $C_X$ and
$c_j$ be a cutting point of $A$ contained in $A\setminus A^{\mathrm{red}}$.
This is a cutting point of the region $R_j$ containing the corresponding terminal point of $A$. Let $x_{f_j}$ be the meridian of the forward strand at $c_j$ and $x_{b_j}$ be that of the backward strand at $c_j$. Let $y_{f_j,l}$ and $y_{f_j,r}$ be the meridians of the regions on the left and right of the forward regions, respectively, and $y_{b_j,l}$ and $y_{b_j,r}$ be those of the backward regions.

\begin{lemma}\label{lemma35_added}
Suppose that $Y\cap \partial D=\emptyset$. 
Let $c_j$ be a cutting point of $A$ contained in $A\setminus A^{\mathrm{red}}$.
Then the identities $y_{f_j,l}=y_{b_j,l}$, $y_{f_j,r}=y_{b_j,r}$ and $x_{f_j}=x_{b_j}$ hold.
\end{lemma}

\begin{proof}
Suppose that the region adjacent to $\partial D$ is on the left of the edge $e_j$ of $C_X$ on which $c_j$ lies. Then, since the meridian of that region is the identity, we have $y_{f_j,l}=y_{b_j,l}=1$. Furthermore, since the terminal point of the cutting tree lies in the region $R_j$ on the right of $e_j$, we have $y_{f_j,r}=y_{b_j,r}$. Hence, from the relations $y_{f_j,r}x_{f_j}y_{f_j,l}^{-1}=1$ and ${y_{b_j,r}x_{b_j}y_{b_j,l}}^{-1}=1$, we have $y_{f_j,r}=x_{f_j}^{-1}=x_{b_j}^{-1}$.

If the region adjacent to $\partial D$ is on the right of $e_j$, we have $y_{f_j,r}=y_{b_j,r}=1$ and  $y_{f_j,l}=y_{b_j,l}$ by the same reason. Hence, from the relations $y_{f_j,r}x_{f_j}y_{f_j,l}^{-1}=1$ and $y_{b_j,r}x_{b_j}{y_{b_j,l}}^{-1}=1$, we have $y_{f_j,l}=x_{f_j}=x_{b_j}$.

Thus, in either case, we obtain the identities in the assertion.
\end{proof}

Now, for an oriented link diagram presentation $D_X$,
the meridians of the strands of $D_X\setminus A^{\mathrm{red}}$ and the regions of $X\setminus A^{\mathrm{red}}$ on $D$ are defined in the same manner as those of $D_X\setminus A$ and $X\setminus A$ on $D$, respectively. These meridians are well-defined due to Lemma~\ref{lemma35_added}.

\begin{theorem}\label{thm2}
Let $(X,\gl)$ be a shadowed polyhedron with 
an oriented link diagram presentation $D_X$.
Let $Y$ be a simple subpolyhedron of $X$ satisfying $X\setminus D\subset Y$ and $Y\cap \partial D=\emptyset$. 
Then, for a reduced system $A^{\mathrm{red}}$ of cutting trees of $D_X$, we have 
\[ 
\pi_1(B^4\setminus Y)\cong
\langle
x_1, \ldots, x_n, y_1,\ldots,y_m \mid
s_1,\ldots, s_{n'},
t_1,\ldots, t_{n''}
\rangle,
\]
where $x_1,\ldots, x_n$ are the meridians of the strands of $D_X\setminus A^{\mathrm{red}}$, $y_1,\ldots, y_m$ are the meridians of the regions of $Y\setminus A^{\mathrm{red}}$ on $D$, 
$s_i=y_{r_i}x_iy_{l_i}^{-1}$ is the
relator obtained for
each edge $e_i$ of $C_X\setminus A^{\mathrm{red}}$, where $y_{l_i}=1$ (resp. $y_{r_i}=1$) if the region on the left (resp. right) of $e_i$ is not contained in $Y$,
and 
$t_j=\gamma_j x_{f_j} \gamma_j^{-1}x_{b_j}^{-1}$ is the relator obtained for each cutting point $c_j$ on $A^{\mathrm{red}}$,
where $x_{f_j}$ and $x_{b_j}$, $\gamma_j$ are the same as those in Theorem~\ref{thm1}. 
\end{theorem}

\begin{proof}
Applying Theorem~\ref{thm1}, we obtain a presentation of the fundamental group $\pi_1(B^4\setminus Y)$
using a system $A$ of cutting trees. We replace $A$ by the reduced system $A^{\mathrm{red}}$ by removing all cutting trees of $A$ intersecting $C_X$ only once. 
Let $c_j$ be a cutting point of $A$ contained in $A\setminus A^{\mathrm{red}}$.
By Lemma~\ref{lemma35_added}, 
the identities $y_{f_j,l}=y_{b_j,l}$, $y_{f_j,r}=y_{b_j,r}$ and $x_{f_j}=x_{b_j}$ hold. 
The relation $\gamma_j x_{f_j}\gamma_j^{-1}x_{b_j}^{-1}=1$ for this cutting point $c_j$ 
can be obtained from $\gamma_j=y_{f_j,r}^{\gl(R_j)-c(R_j)}$  and $y_{f_j,r}=x_{f_j}^{-1}=x_{b_j}^{-1}$ if $y_{f_j,l}=1$, 
and
$\gamma_j=y_{f_j,l}^{\gl(R_j)-c(R_j)}$ and $y_{f_j,l}=x_{f_j}=x_{b_j}$ if $y_{f_j,r}=1$, 
which means that we can remove the relator $\gamma_j x_{f_j}\gamma_j^{-1}x_{b_j}^{-1}$
for this cutting point $c_j$ from the list of relators. 
This completes the proof.
\end{proof}

\section{Wirtinger presentation}

Let $(X, \mathfrak{gl})$ be a shadowed polyhedron with an oriented link diagram presentation $D_X$, and $Y$ be the simple subpolyhedron of $X$ obtained from $X$ by removing the region containing $\partial D$. 
Suppose that $\mathfrak{gl} (R_i) = c(R_i)$ for $i=1, \ldots, n_0 - 1$, where we recall that $n_0$ is the number 
of regions of $X$ on $D$, and 
$\mathfrak{gl} (~ \cdot ~)$ and $c( ~ \cdot ~)$ are defined only for the $n_0 - 1$ regions of $X$ on $D$ that do not contain $\partial D$. 
Let 
\[ \langle x_1, \ldots , x_n, y_1, \ldots , y_m  \mid s_1, \ldots, s_{n'}, t_1, \ldots, t_{n''} \rangle\]
be the presentation of $\pi_1 (B^4 \setminus Y)$ obtained by using a system $A$ of cutting trees as in Theorem 3.1. 
By~\cite{Tur94}, $D_X$ can be regarded as a diagram $D_L$ of the (oriented) link $L=\partial X\setminus \partial D$ in $\partial B^4=S^3$. 
Let 
\[ \langle a_1, \ldots , a_N \mid \mbox{$a_{j_c} a_{i_c} a_{k_c}^{-1} a_{i_c}^{-1}$ for each crossing $c$ of $D_L$}\rangle \]
be the Wirtinger presentation of $\pi_1 (S^3\setminus L)$ obtained by using $D_L$,
where $a_{i_c}$, $a_{j_c}$ and $a_{k_c}$
are the meridians of $L$ around each crossing $c$ of $D_L$ given as in Figure~\ref{fig10}.

\begin{figure}[htbp]
\begin{center}
\includegraphics[width=3cm, bb=231 622 318 708]{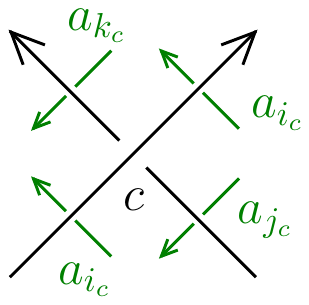}
\caption{Wirtinger presentation.}\label{fig10}
\end{center}
\end{figure}

We define a homomorphism $\Phi$ from the free group $\langle a_1, \ldots , a_N \rangle$ to 
the free group  $\langle x_1, \ldots , x_n, y_1, \ldots , y_m \rangle$ as follows: 
\begin{enumerate}
\item
For a generator $a_j$ corresponding to the strand of $D_L$ that does not intersect the system $A$ of cutting trees, 
define $\Phi (a_j)$ to be the generator of the latter free group corresponding to  the meridian of that strand of $D_X \setminus A$. 
\item
For a generator $a_j$ corresponding to the strand of $D_L$ that  intersects $A$ at the cutting point $c_j$, 
define $\Phi (a_j)$ to be the generator of the latter free group corresponding to  the meridian $x_{f_j}$ of the forward strand 
at $c_j$. 
\end{enumerate}

\begin{theorem}\label{thm3}
In the above setting, the map $\Phi$ induces an isomorphism 
$\Phi' : \pi_1 ( S^3 \setminus L ) \to \pi_1 (B^4 \setminus Y)$ between the quotient groups.
\end{theorem}

\begin{proof}
In order to clarify the arguments, in this proof we use symbols $x_j$, $y_j$ and $a_j$ for elements of the free groups, and 
their equivalence classes in $\pi_1 ( S^3 \setminus L )$
and $\pi_1 (B^4 \setminus Y)$ are denoted by using the square brackets $[ ~ \cdot ~]$. 
Since $\mathfrak{gl} (R_i) = c(R_i) $ for $i=1, \ldots, n_0 - 1$, we have $[\gamma_j] = 1$ for each cutting point $c_j$ by~\eqref{eq103}, 
and hence $[x_{f_j}] = [x_{b_j}]$ by the relator $t_j$.

We first show that $\Phi'$ is well-defined. 
Let $[y_1], [y_2], [y_3]$ and $[y_4]$ be the meridians of the regions of $X$ on $D$ adjacent to the vertex of $X$ corresponding 
to a crossing $c$ of the diagram $D_X$ given as in Figure~\ref{fig11}. 
From the relators $s_1, \ldots, s_{n'}$ in Theorem~\ref{thm2}, we have 
\[
[y_4 x_{i_c} y_3^{-1}] = [y_1 x_{i_c} y_2^{-1}] = [y_1 x_{j_c} y_4^{-1}] = [y_2 x_{k_c} y_3^{-1}] = 1 
\]
in $\pi_1 (B^4 \setminus Y)$. 
Thus, we have 
\[
\begin{split}
[\Phi (a_{j_c} a_{i_c} a_{k_c}^{-1} a_{i_c}^{-1})] &= [x_{j_c} x_{i_c} x_{k_c}^{-1} x_{i_c}^{-1}] \\
&= [y_1^{-1} ( y_1 x_{j_c} y_4^{-1}) ( y_4 x_{i_c} y_3^{-1}) ( y_2 x_{k_c} y_3^{-1})^{-1}
( y_1 x_{i_c} y_2^{-1})^{-1} y_1] = 1,  
\end{split}
\]
which implies that $\Phi'$ is well-defined.

\begin{figure}[htbp]
\begin{center}
\includegraphics[width=3.5cm, bb=220 606 330 708]{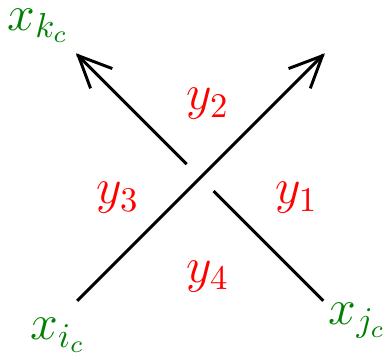}
\caption{The meridians of regions around a crossing.}\label{fig11}
\end{center}
\end{figure}

Next, we show that $\Phi'$ is an epimorphism. 
Since the subpolyhedron $Y$ does not contain the region containing $\partial D$, the meridians of the regions of 
$D_X \setminus A$ that intersect $\partial D$ are all the identity element of $\pi_1 (B^4 \setminus Y)$. 
Thus, for each $j=1,\ldots,m$,
we obtain a word $w_j(x_1, \ldots, x_n)$ in $\{ x_1, \ldots, x_n \}$ 
satisfying 
$[w_j(x_1, \ldots, x_n)] = [y_j]$ 
by using the relators $s_1, \ldots, s_{n'}$,
which implies that $\Phi'$ is an epimorphism. 
More precisely, we can show this fact by using those relators corresponding to strands of $D_X \setminus A$
meeting a simple path in $D \setminus A$ from a point on the region of $D_X \setminus A$ 
with the meridian $y_j$ to a point on a region of $D_X \setminus A$ intersecting $\partial D$, 
but we omit the details here for simplicity of exposition.

Finally, we show that $\Phi'$ is a monomorphism. 
To show this, we define a homomorphism $\Psi$ from 
the free group  $\langle x_1, \ldots , x_n, y_1, \ldots , y_m \rangle$ to 
the free group $\langle a_1, \ldots , a_N \rangle$ as follows: 
\begin{enumerate}
\item
For a generator $x_i$ corresponding to the meridian of a strand of $D_X \setminus A$, 
define $\Psi (x_i)$ to be the generator of  $\langle a_1, \ldots , a_N \rangle$ corresponding to the strand of $D_L$ 
containing that strand of $D_X \setminus A$. 
\item
For a generator $y_j$ corresponding to the meridian of a region of $D_X \setminus A$, 
first fix a word $w_j(x_1, \ldots, x_n)$ in $\{ x_1, \ldots, x_n \}$ satisfying 
$[w_j(x_1, \ldots, x_n)] = [y_j]$, and then define $\Psi (y_j)$ by 
$\Psi (y_j) = \Psi (w_j(x_1, \ldots, x_n))$. 
The existence of such a word $w_j(x_1, \ldots, x_n)$ follows from the above argument. 
\end{enumerate}
Now, by definition, we can easily check that 
$\Psi$ induces a well-defined homomorphism 
$\Psi' : \pi_1 (B^4 \setminus Y) \to \pi_1 ( S^3 \setminus L )$ and 
$\Psi' \circ \Phi'$ is the identity on $\pi_1 ( S^3 \setminus L )$, which implies that 
$\Phi'$ is a monomorphism.
\end{proof}

\begin{remark}\label{rem100}
Let $(X,\gl)$ be a shadowed polyhedron with a link diagram presentation $D_X$
and $Y$ be the simple subpolyhedron of $X$ obtained from $X$ by removing the region containing $\partial D$ as in Theorem~\ref{thm3}.  
Suppose that $D$ is the unit disk on $\Real^2$.
Since $X$ collapses onto $D$, we can identify $B^4$, in which $X$ is embedded, with $D\times D^2$ so that the disk $D$ lies as $D\times \{(0,0)\}\subset X\subset B^4$, where $D^2$ is the unit disk on $\Real^2$.
Choose $\varepsilon>0$ sufficiently small so that
$\Nbd(\partial D; D)=\{u\in D\mid |u|\geq 1-\varepsilon\}$ does not intersect $\Sing(X)$, and 
consider the homotopy $\varphi_t$ from $D\times D^2$ to itself given by 
\[
\varphi_t(u,u')=
\begin{cases}
  ((1-(1-|u'|)t)u, u')\in D\times D^2
  & \text{if $u\not\in \Nbd(\partial D; D)$} \\
  \left(\left(1-\frac{(1-|u|)(1-|u'|)}{\varepsilon}t\right)u, u'\right) \in D\times D^2 
  & \text{if $u\in \Nbd(\partial D; D)$}.
\end{cases}
\]
This homotopy shows that 
$B^4\setminus Y$ is homotopy-equivalent to $B^4\setminus \varphi_1(Y)$. Moreover, the pair $(B^4, \varphi_1(Y))$ is homeomorphic to the cone of the pair $(S^3, L)$,
where $L$ is the link given by the diagram $D_X$.
This allows us to show the isomorphism of the two fundamental groups in Theorem~\ref{thm3} without comparing their presentations.
\end{remark}

\section{Lefschetz fibrations of divides}

\begin{definition}\label{dfn21}
A {\it divide} $P$ is the image of a generic and relative immersion of 
a finite number of copies of the unit interval or the unit circle into
the unit disk $D$ on $\Real^2$. The generic condition is that
\begin{itemize}
   \item the self-intersection points in the image lie in the interior of $D$ and are only normal crossings,
   \item an immersed interval intersects $\partial D$ at the endpoints transversely, and
   \item an immersed circle does not intersect $\partial D$.
\end{itemize}
\end{definition}

A divide was introduced by N.\,A'Campo in~\cite{AC99,AC98} as a generalization of real morsified curves of complex plane curve singularities~\cite{AC75a, AC75b, GZ74a, GZ74b, GZ77}.

Let $P$ be a divide on $D$ and regard the tangent bundle $T\Real^2$ of the real $2$-plane $\Real^2$ 
as the $4$-dimensional space $\Real^4$ in a natural way.
The link $L(P)$ of $P$ is the set of simple closed curves in the unit sphere $S^3$ in $\Real^4=T\Real^2$ defined by
\[
    L(P)=\{(x,u)\in \Real^4\mid x\in P,\; u\in T_xP,\; \|(x,u)\|=1\}.
\]
In~\cite{AC98}, A'Campo proved that if a divide is connected then its link is fibered whose monodromy is a product of right-handed Dehn twists.
He also proved that if a divide is a real morsified curve of a complex plane curve singularity then its fibration is isomorphic to the Milnor fibration, and if a divide consists only of immersed intervals then the unknotting number of the link is equal to the number of double points. Furthermore, in~\cite{AC98b}, he proved that there are many links of divides that are hyperbolic. The link-types of the links of divides had been studied by  Couture-Perron~\cite{CP00}, Hirasawa~\cite{Hir02} 
and Kawamura~\cite{Kaw02}, 
and there are many related works (see for instance the references in~\cite{IN20}).

\begin{figure}[htbp]
\includegraphics[scale=0.6, bb=129 544 533 713]{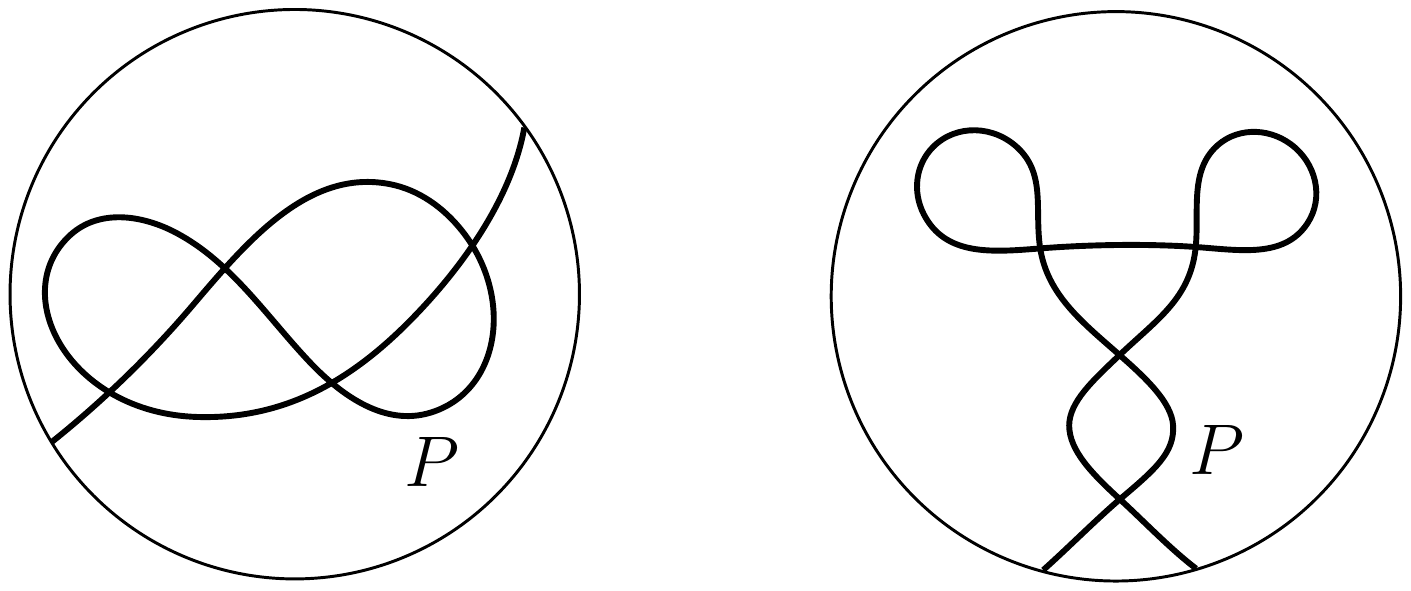}
\caption{Divides in the unit disk: The left is a divide of the $(3,5)$-torus knot, which is a real morsified curve of the singularity of $f(z,w)=z^3-w^5$. The right one does not come from a singularity. The link of this divide is $10_{139}$.}
\label{fig14}
\end{figure}

Let $P$ be a connected divide.
Let $f_P:\Real^2\to\Real$ 
be a Morse function on $\Real^2$ 
such that $f_P^{-1}(0)\cap D=P$ and each region bounded by $P$ has only one Morse singularity. The fibration of a divide is given by the map 
\[
   \frac{F_P}{|F_P|}: S^3\setminus L(P)\to S^1,
\]
where $F_P:T\Real^2\to\Complex$ is a kind of complexified function of $f_P$ given as
\[
   F_P(x,u)=f_P(x)+\sqrt{-1}\eta df_P(x)(u)-\frac{1}{2}\eta^2\chi(x)H_{f_P}(x)(u,u),
\]
where $\eta$ is a sufficiently small positive real number, $H_{f_P}$ is the Hessian of $f_P$ and $\chi(x)$ is a bump function which is $1$ at the double points of $P$ and $0$ outside small neighborhoods of the double points. 
The intersection of $F_P^{-1}(0)$ with $S^3$ coincides with the link $L(P)$ of the divide $P$.

The singular values of the map $F_P$ lie on the real line in the target $\Complex$. We choose a narrow disk $D_\varepsilon$ in $\Complex$ containing all singular values.
Then the map $F_P:B^4\cap F_P^{-1}(D_\varepsilon)\to\Complex$ is a Lefschetz fibration and the fibration of $S^3\setminus L(P)$
can be regarded as the restriction of $F_P$ to $B^4\cap F_P^{-1}(\partial D_\varepsilon)$.
Focusing on this property, a divide and its fibration had been generalized to those on a compact oriented surface in~\cite{Ishi04}. 

A doubling method was introduced by W.~Gibson and the first author in~\cite{GI02a, GI02b} to obtain a link diagram of a divide. Recently,  in~\cite{IN20}, the first and the third author clarified the relationship between divides and shadowed polyhedrons via the doubling method. 
In their paper, the doubled curve of a divide $P$ is obtained by the following steps:
\begin{itemize}
\item[1.] Double the curve of $P$.
\item[2.] For each endpoint of $P$, close the corresponding two endpoints of the doubled curve by a small half circle.
\item[3.] For each edge of $P$ that is not adjacent to an endpoint, add a crossing between the  two edges of the doubled curve parallel to the edge.
\end{itemize}
Figure~\ref{fig13} is an example of a divide and its doubled curve. 
The immersed curve on the right is the doubled curve of the divide described on the left. 

\begin{figure}[htbp]
\includegraphics[scale=0.7, bb=129 543 565 713]{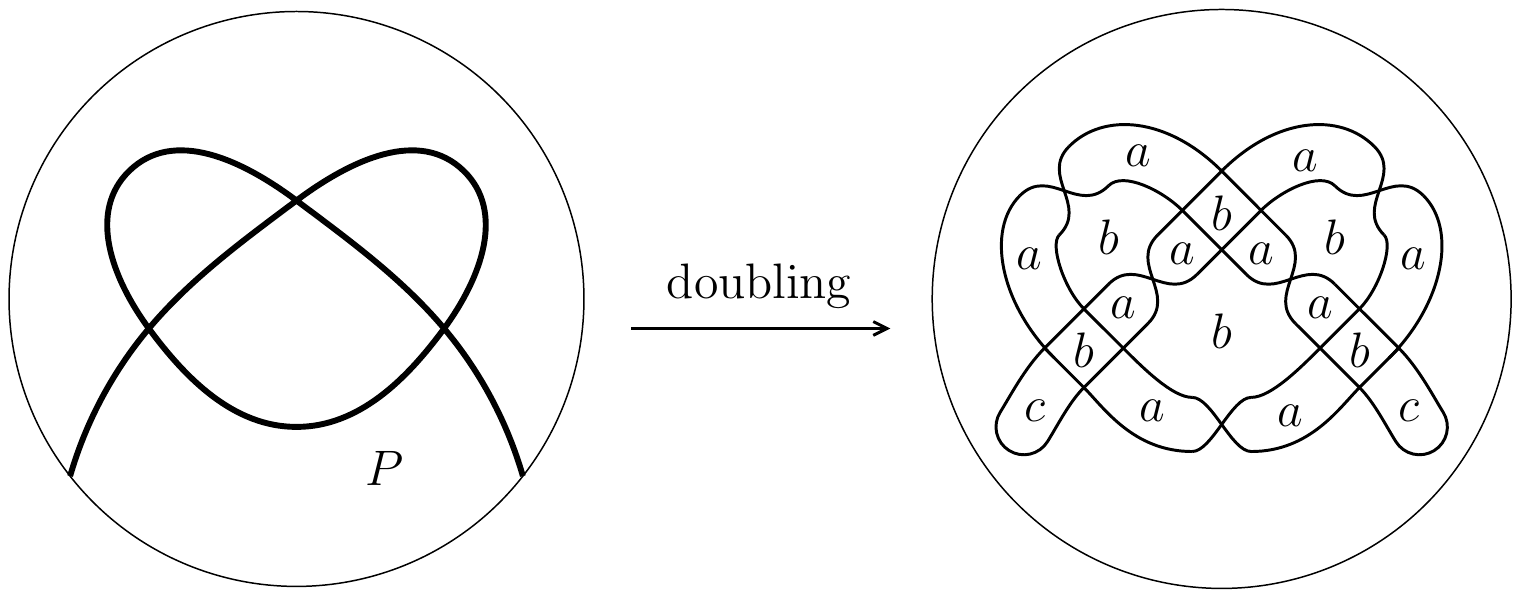}
\caption{The figure on the right is an immersed curve presentation of the shadowed polyhedron of the divide on the left. The gleams of the internal regions labeled by $a$, $b$, $c$ are $\frac{1}{2}$, $-1$, $0$, respectively. The link of this divide is the $(3,4)$-torus knot.}
\label{fig13}
\end{figure}

Each edge both of whose endpoints are vertices of $P$ corresponds to two triangular regions bounded by the doubled curve, and we assign the label $a$ to these regions. 
Each edge one of whose endpoint is not a vertex of $P$ corresponds to a bigon bounded by the double curve, and we assign the label $c$ to this region.
Each region bounded by $P$ corresponds to a region bounded by the doubled curve, 
and each double point of $P$ corresponds to a square region bounded by the doubled curve.
We assign the label $b$ to these regions. 
We assign half-integers $\frac{1}{2}$, $-1$, $0$ as gleams to the regions labeled by $a$, $b$, $c$, respectively,
and regard the doubled curve as an immersed curve presentation of a shadowed polyhedron. 

\begin{definition}
The shadowed polyhedron $(X,\gl)$ obtained from a divide $P$ as above is called the {\it shadowed polyhedron of a divide $P$}.
\end{definition}

Let $(X_P,\gl_P)$ be the shadowed polyhedron of a connected divide $P$. 
The $4$-manifold of $(X_P,\gl_P)$ is a $4$-ball, which is regarded as the unit $4$-ball in $T\Real^2$. 
The assertion in~\cite{IN20} is that a regular fiber of the Lefschetz fibration of a connected divide $P$, embedded in $B^4$, is the closure $Y$ of the union of the regions labeled by $a$ and $c$ and the annular regions $X\setminus D$. Moreover, each vanishing cycle, which is the core of a right-handed Dehn twist of a complex Morse singularity of the Lefschetz fibration, is the boundary of a region labeled by $b$ and it vanishes to the center of that region. In particular, if a divide is obtained from a real morsification of an isolated, complex plane curve singularity then $B^4$ is regarded as the Milnor ball and $Y$ is regarded as a Milnor fiber embedded in $B^4$.
Thus the embedding of a regular fiber of the Lefschetz fibration of a divide, including a Milnor fiber,
can be completely described by using the shadowed polyhedron. 
Due to this description, by applying Theorem~\ref{thm2}, we can calculate the fundamental groups of the complements of fibers of the Lefschetz fibrations of divides in $B^4$.

\begin{theorem}\label{thm4}
Let $(X_P,\gl_P)$ be the shadowed polyhedron of a connected divide $P$.
\begin{itemize}
\item[(1)] Let $Y_{ac}$ be the closure of the union of the regions labeled by $a$ and $c$ and the annular regions $X\setminus D$. 
Then the pair $(B^4, Y_{ac})$ with rounding corners of $Y_{ac}$ is diffeomorphic to the pair $(B^4, B^4\cap F_P^{-1}(t))$, where $t$ is a regular value on $\Complex$ sufficiently close to the origin.
\item[(2)] Let $Y_{a\square c}$ be the union of $Y_{ac}$ and the square regions corresponding to the double points of $P$. Then 
$B^4\setminus Y_{a\square c}$ is homotopy-equivalent to $B^4\setminus F_P^{-1}(0)$.
\item[(3)]  Let $f:\Complex^2\to\Complex$ be a polynomial map with an isolated singularity at the origin and $f(0,0)=0$ and $B^4$ be a sufficiently small ball in $\Complex^2$ centered at the origin.
Let $P$ be a divide obtained from the singularity of $f$ at the origin by a real morsification
and $Y_{abc}$ be the subpolyhedron of $X_P$ obtained from $X_P$ by removing the region containing $\partial D$. 
Then $B^4\setminus Y_{abc}$ is homotopy-equivalent to $B^4\setminus f^{-1}(0)$.
\end{itemize}
\end{theorem}

\begin{proof}
The assertion~(1) follows from the observation in~\cite{IN20}. 
The Milnor fiber of a complex Morse singularity is an annulus, and the singular fiber, which consists of two complex planes intersecting at their origins, is obtained from the annulus 
by shrinking the vanishing cycle to the singular point. 
Therefore, the assertion~(2) follows. Note that we can shrink these vanishing cycles independently since they are disjoint on $Y_{ac}$.
In the case of the assertion~(3), we need to shrink all vanishing cycles.
This is possible if $P$ is obtained from a real morsification since the reverse operation of the morsification ensures the existence of simultaneous shrinking of all the vanishing cycles.
\end{proof}

\begin{remark}
For an oriented divide $\vec P$ on the unit disk $D$, we can make its shadowed polyhedron $(X_{\vec P}, \mathfrak{gl}_{\vec P})$ easily, 
where $X_{\vec P}$ is the shadow and $\mathfrak{gl}_{\vec P}$ is its gleam function.
See~\cite[Lemma 3.1]{IN20}. 
The shadow $X_{\vec P}$ is the union of $D$ and annuli like the shadows of divides.
Let $Y$ be the subpolyhedron of $X_{\vec P}$ obtained by removing the region containing $\partial D$.
A link in $S^3$ is defined for each oriented divide (see \cite{GI02a} for the definition), called the link of an oriented divide, and this is isotopic to $\partial Y$ in $S^3=\partial B^4$. 
As explained in Remark~\ref{rem100},
$\pi_1(B^4\setminus Y)$ is isomorphic to $\pi_1(S^3\setminus \partial Y)$.
Hence the fundamental group of the complement of the link of an oriented divide $\vec P$ can be calculated 
by applying Theorem~\ref{thm2} to $Y\subset X_{\vec P}$.
\end{remark}

\begin{example}\label{ex1}
Let $P$ be a divide on the left-top in Figure~\ref{fig15}. Its doubled curve is described on the left-bottom.
This divide is the real part of the Morse singularity of $f:\Complex^2\to\Complex$ given by $f(z,w)=z^2-w^2$. 
We assign over/under information to the double points of the doubled curve as shown in the figure on the right
and regard it as a link diagram presentation of the shadowed polyhedron $X_P$ of $P$.
Orient the strands and choose a reduced system of cutting trees as in the figure,
and set the generators $x$ and $y^{-1}$ to be the meridians of the right-top and right-bottom regions, respectively.

\begin{figure}[htbp]
\includegraphics[width=13.5cm, bb=129 483 466 712]{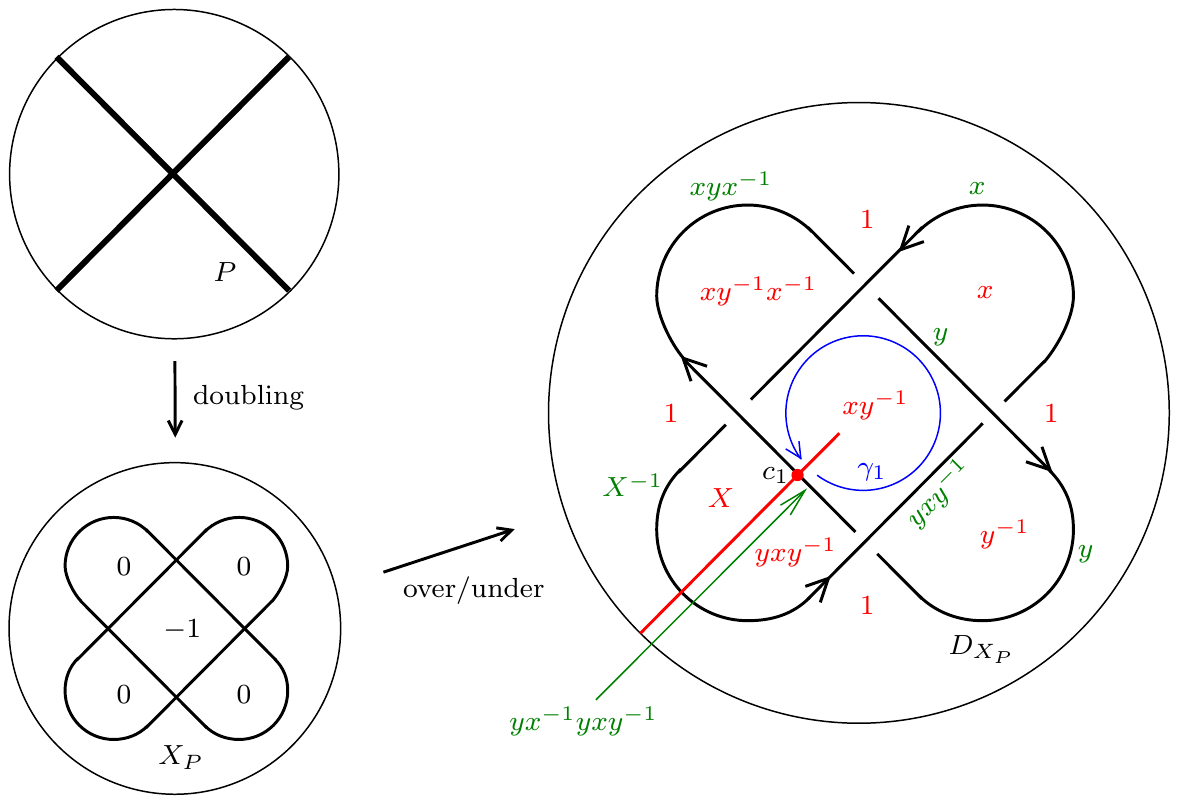}
\caption{A link diagram presentation of the shadowed polyhedron of a divide of a complex Morse singularity.  In the figure, $X=xy^{-1}xyx^{-1}$.}
\label{fig15}
\end{figure}

\begin{itemize}
\item[(1)] Let $Y$ be the subpolyhedron of $X_P$ obtained from $X_P$ by removing the region containing $\partial D$. We calculate $\pi_1(B^4\setminus Y)$. Since the meridian of the region containing $\partial D$ is the identity, we first write $1$ on the region. Next, we calculate the other meridians 
by using the relations $y_{r_i}x_iy_{l_i}^{-1}=1$ in Theorem~\ref{thm3} inductively. 
Finally, we use the relation~\eqref{eq103} in Theorem~\ref{thm3}. 
Let $c_1$ be the cutting point of the region with gleam $-1$. 
Then the relation~\eqref{eq103} is written as $\gamma_1x_{f_1}\gamma_1^{-1}x_{b_1}^{-1}=1$. 
Since $\gamma_1=xy^{-1}$, we have
\[
(xy^{-1})(yx^{-1}yxy^{-1})(xy^{-1})^{-1}(xyx^{-1})^{-1}
=yxy^{-1}x^{-1}=1.
\]
We can verify that the relation obtained from the other cutting point is also $yxy^{-1}x^{-1}=1$.
Thus, by Theorem~\ref{thm4}~(3), we have $\pi_1(B^4\setminus f^{-1}(0))=\Integer\langle x\rangle\oplus \Integer\langle y\rangle$. Since $(B^4,B^4\cap f^{-1}(0))$ is the cone of $(S^3, \text{a Hopf link})$, this coincides with the fundamental group of the complement of a Hopf link.
\item[(2)] Let $Y_{ac}$ be the subpolyhedron obtained from $Y$ by removing the region with gleam $-1$. The Milnor fiber of the Morse singularity is an annulus and we can see it directly on $X_P$ as $Y_{ac}$.
We now calculate $\pi_1(B^4\setminus Y_{ac})$. 
Since the meridian of the region with gleam $-1$ is the identity, we have the relation $xy^{-1}=1$ additionally. Hence 
\[
   \pi_1(B^4\setminus Y_{ac})\cong\langle x, y \mid xy=yx, \; 
   xy^{-1}=1\rangle\cong \Integer.
\]
\item[(3)] 
Set the gleams of the digonal regions to be $1$ and those of the square region to be $-2$ 
so that it satisfies the condition $\gl(R_j)=c(R_j)$ in Theorem~\ref{thm3}. Then $\gamma_j=1$ and we have
the relation $yx^{-1}yxy^{-1}=xyx^{-1}$. Setting ${\mathcal Y}=xyx^{-1}$, we have
\[
   \pi_1(B^4\setminus Y)\cong \langle x, {\mathcal Y} \mid 
x^{-1}{\mathcal Y}x^{-1}{\mathcal Y}x{\mathcal Y}^{-1}x{\mathcal Y}^{-1}=1\rangle,
\] 
which is the fundamental group of the complement of a $(2,4)$-torus link, which is the link given by the diagram $D_{X_P}$.
\end{itemize}
\end{example}

\begin{example}
Let $P$ be a divide on the left-top in Figure~\ref{fig16}. Its doubled curve is described on the left-bottom.
This divide is the real part of a real morsified curve of the singularity of $f(z,w)=z^2-w^3$. 
We assign over/under information to the double points of the doubled curve as shown in the figure on the right
and regard it as a link diagram presentation of the shadowed polyhedron $X_P$ of $P$.
Orient the strands and choose a reduced system of cutting trees as in the figure,
and set the generators $x$ and $y^{-1}$ to be the meridians of the right-top and right-bottom regions, respectively.

\begin{figure}[htbp]
\includegraphics[width=13.5cm, bb=129 483 503 712]{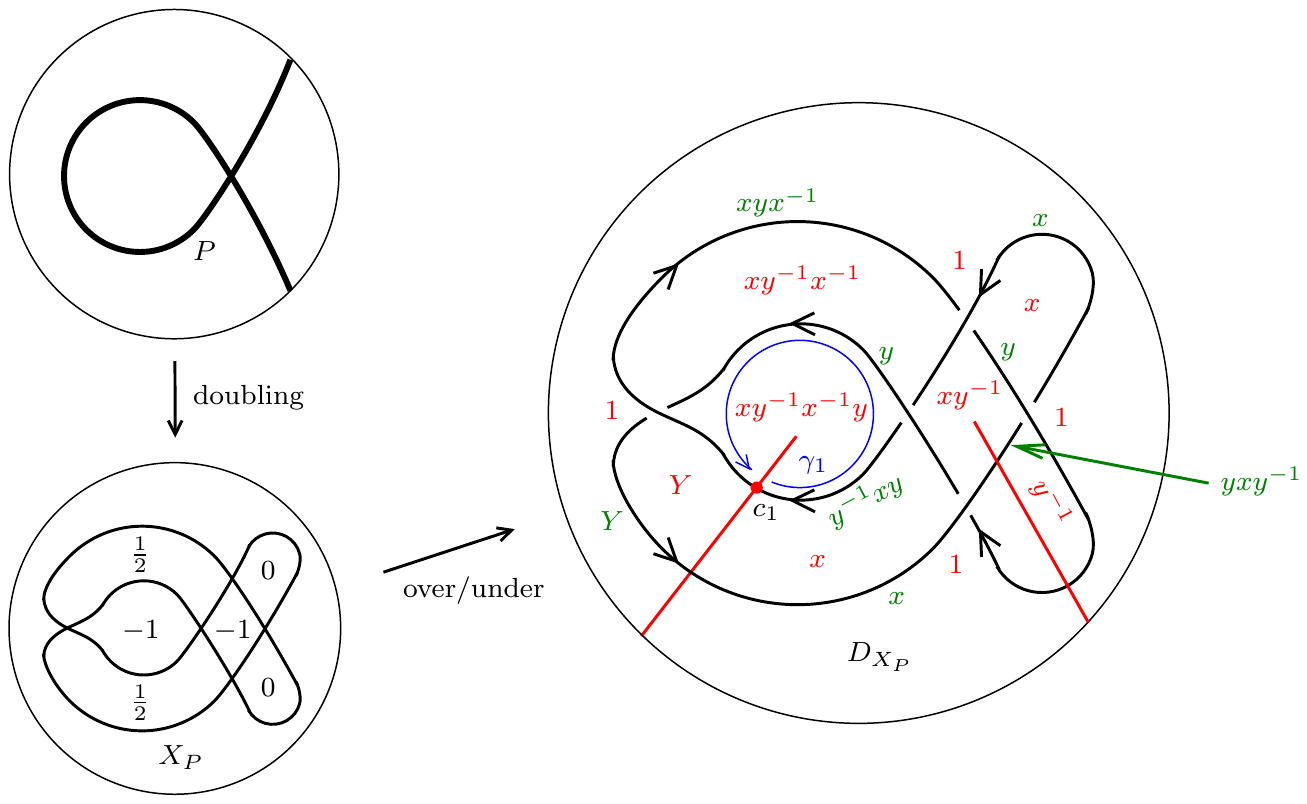}
\caption{A link diagram presentation of the shadowed polyhedron of a divide of the singularity of $f(z,w)=z^2-w^3$. In the figure, $Y=xy^{-1}x^{-1}yxyx^{-1}$.}\label{fig16}
\end{figure}

\begin{itemize}
\item[(1)] Let $Y$ be the subpolyhedron of $X_P$ obtained from $X_P$ by removing the region containing $\partial D$. We calculate $\pi_1(B^4\setminus Y)$ as we did in Example~\ref{ex1}~(1). 
Let $c_1$ be the cutting point, on the left cutting tree, of the region with gleam $-1$. 
Since $\gamma_1=1$, we have the relation $y^{-1}xy=xyx^{-1}$.
We can verify that the relation obtained from the other cutting point on the left cutting tree is also $y^{-1}xy=xyx^{-1}$. 
Thus, by Theorem~\ref{thm4}~(3), we have 
\[
   \pi_1(B^4\setminus Y)\cong\langle x, y \mid xyx=yxy \rangle.
\]
Since $(B^4,B^4\cap f^{-1}(0))$ is the cone of $(S^3, \text{a $(2,3)$-torus knot})$, this coincides with the fundamental group of the complement of a $(2,3)$-torus knot.
\item[(2)] Let $Y_{ac}$ be the subpolyhedron of $X_P$ obtained from $Y$ by removing the region with gleam $-1$. The Milnor fiber of the singularity of $f(z,w)=z^2-w^3$ is a torus with one boundary component, and we can see it directly on $X_P$ as $Y_{ac}$.
Since the meridian of the region with gleam $-1$ is the identity, we have $xy^{-1}=1$. 
Hence
$\pi_1(B^4\setminus Y_{ac})\cong \Integer$. 
\item[(3)] Let $Y_{b_i}$ be the subpolyhedron of $X_P$ obtained from $Y$ be removing only one of the two region with gleam $-1$. In either case, we have $x=y$ and hence $\pi_1(B^4\setminus Y_{b_i})\cong \Integer$. 
\end{itemize}
\end{example}

\section{Complexified real line arrangements}

Let $\mathcal A$ be a line arrangement on $\Real^2$, which is the union of $k$ different lines on $\Real^2$.
It is given by the zero set of a polynomial 
\[
   f(u,v)=\prod_{i=1}^k(a_ku+b_kv+c_k),
\]
where $[a_i:b_i:c_i]\ne [a_j:b_j:c_j]$ for $i\ne j$. Regarding the variables $(u,v)$ as complex variables $(z,w)$, we obtain a complexified real line arrangement, denoted by $\mathcal A_{\Complex}$.

The fundamental groups of the complements of complexified real line arrangements can be determined by combinatorics of the arrangements easily, see~\cite{Ran82, Ran82a}.
For further information about the fundamental groups of arrangements, see~\cite{OT92} and the references therein.
Minimal stratifications of the complements of complexified real line arrangements had been studied by Yosihnaga in~\cite{Yos12}.
Recently, Sugawara and Yoshinaga gave Kirby diagrams for those arrangements using divides with cusps~\cite{SY21}.

\begin{example}\label{ex61}
Let $\mathcal A$ be a real line arrangement consisting of $4$ generic lines, described on the left-top in Figure~\ref{fig16}. Its doubled curve is described on the left-bottom.
We assign over/under information to the double points of the doubled curve as the figure on the right
and regard it as a link diagram presentation of the shadowed polyhedron $X_{\mathcal A}$ of $\mathcal A$.
Orient the strands and choose the reduced system $A^{\mathrm{red}}$ of cutting trees
and set the generators $x, y^{-1}, z$ and $w^{-1}$ to be the meridians as in the figure.

\begin{figure}[htbp]
\includegraphics[width=13.5cm, bb=129 474 495 712]{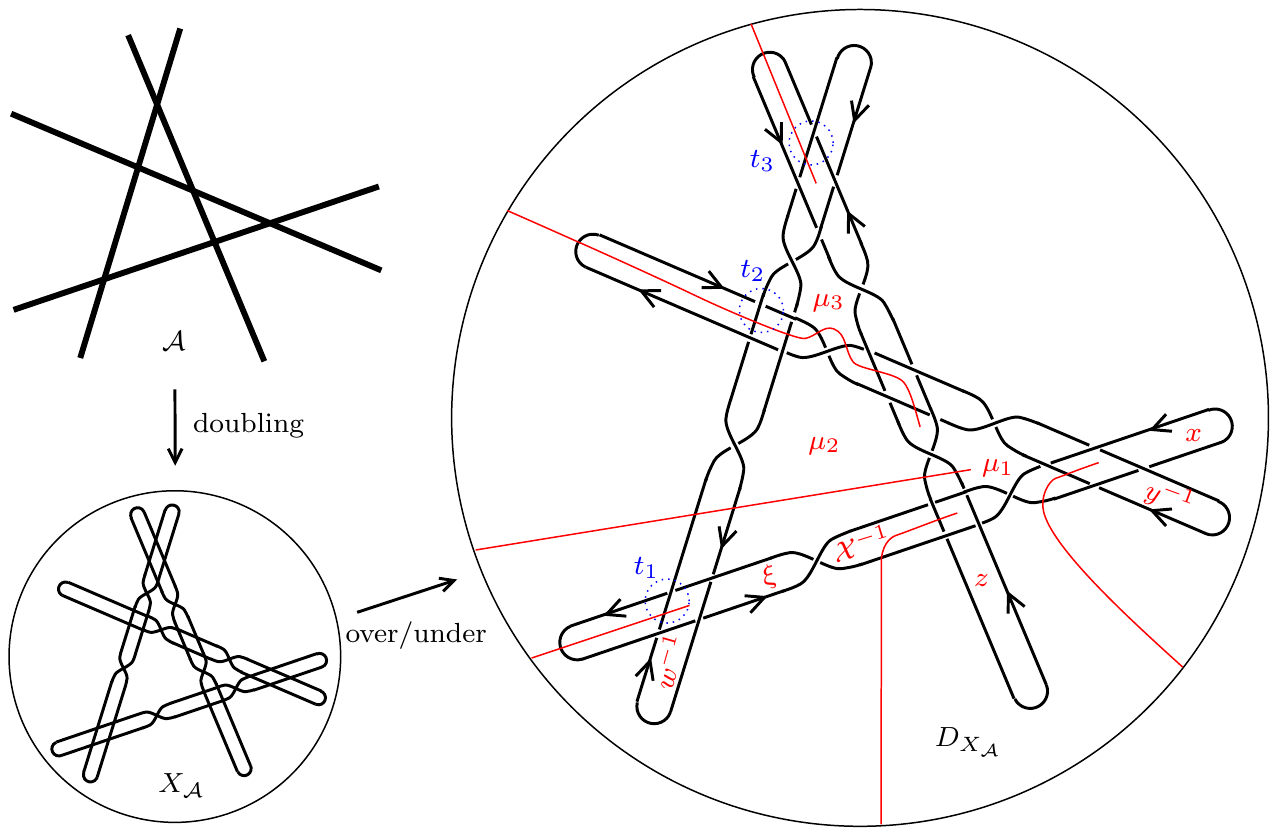}
\caption{A link diagram presentation of the shadowed polyhedron of the real line arrangement $\mathcal A$.}
\label{fig18}
\end{figure}

\begin{itemize}
\item[(1)] Let $Y$ be the subpolyhedron of $X_{\mathcal A}$ obtained from $X_{\mathcal A}$ by removing the region containing $\partial D$. We calculate $\pi_1(B^4\setminus Y)$. 
By applying Theorem~\ref{thm2}, we can fix the meridians of all regions of $X_{\mathcal A}\setminus A^{\mathrm{red}}$  and obtain three relations 
\[
\begin{split}
   t_1:\;\; & 
   \xi^{-1}{\mathcal X}\xi w\xi^{-1}=w \\
   t_2:\;\; & 
   x^{-1}y\xi wzy^{-1}
   =z{\mathcal X}x{\mathcal X}^{-1}z^{-1}wz{\mathcal X}x^{-2} \\
   t_3:\;\; &
   x{\mathcal X}^{-1}z^{-1}wz{\mathcal X}x^{-1}
   ={\mathcal X}^{-1}z{\mathcal X}x{\mathcal X}^{-1}z^{-1}wz{\mathcal X}x^{-1}yz^{-1}y^{-1}
\end{split}
\]
around each of the crossings $t_1, t_2, t_3$ in the figure, 
where ${\mathcal X}=y^{-1}xy$ and $\xi=z{\mathcal X}x{\mathcal X}^{-1}z^{-1}$. 
Setting ${\mathcal Z}=x^{-1}yzy^{-1}x$ and ${\mathcal W}=x{\mathcal X}^{-1}z^{-1}wz{\mathcal X}x^{-1}$, we have
$
x{\mathcal Z}y{\mathcal W} 
=
{\mathcal Z}y{\mathcal W}x
=
y{\mathcal W}x{\mathcal Z}
=
{\mathcal W}x{\mathcal Z}y
$. 
By Theorem~\ref{thm4}~(3), we see that $\pi_1(B^4\setminus Y)$ is the fundamental group of the complement of four complex lines intersecting at one point, which is the complexification $\mathcal A'_{\Complex}$ of the real line arrangement $\mathcal A'$ shown on the left in Figure~\ref{fig19}. Thus we have
\[
\begin{split}
   \pi_1(\Complex^2\setminus \mathcal A'_{\Complex})
   &\cong\pi_1(B^4\setminus Y) \\
   &\cong 
   \langle x,y, {\mathcal Z}, {\mathcal W} \mid 
   x{\mathcal Z}y{\mathcal W} 
=
{\mathcal Z}y{\mathcal W}x
=
y{\mathcal W}x{\mathcal Z}
=
{\mathcal W}x{\mathcal Z}y
\rangle.
\end{split}
\] 
Note that this is the fundamental group of the complement of the singular fiber of $f(z,w)=z^4-w^4$ in the Milnor ball and it coincides with the fundamental group of the complement of a $(4,4)$-torus link in $S^3$.
\begin{figure}[htbp]
\includegraphics[width=13.5cm, bb=129 595 520 713]{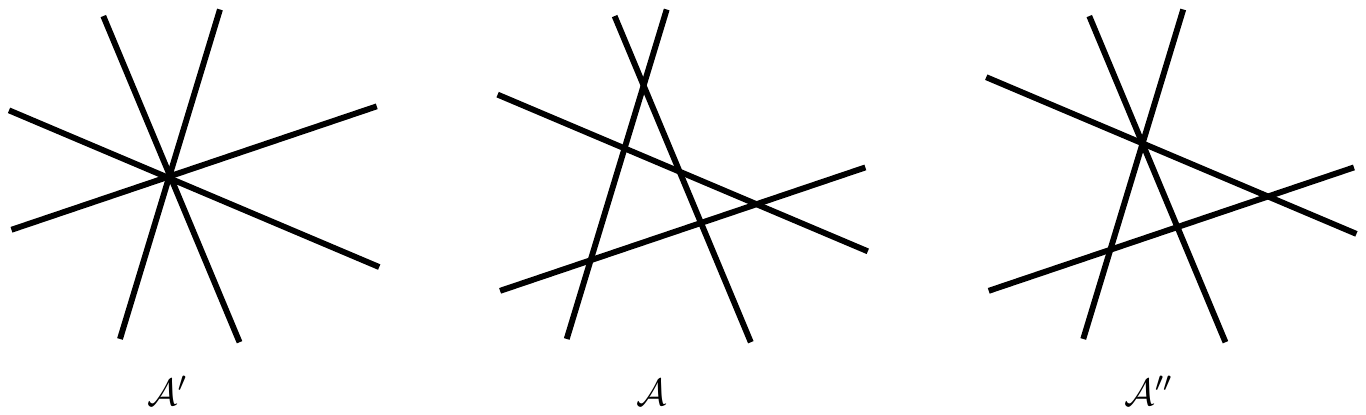}
\caption{Real line arrangements in Example~\ref{ex61}.}
\label{fig19}
\end{figure}
\item[(2)] Let $Y_{a\square c}$ be the subpolyhedron of $X_{\mathcal A}$ obtained from $Y$ by removing 
the regions corresponding to the regions bounded by $\mathcal A$ (in other words the chambers of $\mathcal A$). 
The set $Y_{a\square c}$ is regarded as the complexified real line arrangement $\mathcal A_{\Complex}$
of $\mathcal A$. 
The meridians $\mu_1, \mu_2$ and $\mu_3$ of the regions of $X_{\mathcal A}\setminus A^{\mathrm{red}}$ in the figure are calculated as
\[
\begin{split}
   \mu_1= & x{\mathcal X}^{-1} \\
   \mu_2= & 
   x{\mathcal X}^{-1}x{\mathcal X}^{-1}z^{-1}{\mathcal X}^{-1}z{\mathcal X}^2x^{-1}
   \\
   \mu_3= & 
   xy^{-1}x^{-1}y^2z^{-1}y^{-1}{\mathcal X}^{-1}z{\mathcal X}^2x^{-1}
\end{split}
\]
and they are equal to $1$ in $\pi_1(B^4\setminus Y_{a\square c})$. 
From these relations, we have $xy=yx$, $xz=zx$ and $yz=zy$, and applying them to the relations $t_1, t_2$ and $t_3$ we have $xw=wx$, $yw=wy$ and $zw=wz$. Thus
\[
   \pi_1(\Complex^2\setminus \mathcal A_{\Complex})
   \cong
   \pi_1(B^4\setminus Y_{a\square c}) 
   \cong 
   \Integer\langle x \rangle
   \oplus
   \Integer\langle y \rangle
   \oplus
   \Integer\langle z \rangle
   \oplus
   \Integer\langle w \rangle.
\]
\item[(3)] Let $Y_{b_3}$ be the subpolyhedron of $X_{\mathcal A}$ obtained from $Y_{a\square c}$ by attaching the region with the meridian $\mu_3$. The set $Y_{b_3}$ corresponds to the complexification $\mathcal A''_{\Complex}$ of the real line arrangement $\mathcal A''$ shown on the right in Figure~\ref{fig19}. 
The relations for $\pi_1(B^4\setminus Y_{b_3})$ are $\mu_1=\mu_2=1$ and $t_1, t_2, t_3$. 
From  $\mu_1=\mu_2=1$ we have $xy=yx$ and $xz=zx$.
Applying them to the relations $t_1, t_2$ and $t_3$ we have $xw=wx$, $wzy=ywz$ and $zyw=wzy$. 
Thus
\[
   \pi_1(\Complex^2\setminus \mathcal A''_{\Complex})
   \cong
   \pi_1(B^4\setminus Y_{b_3}) 
   \cong 
   \Integer\langle x \rangle
   \oplus
   \langle y,z,w\mid 
   wzy=ywz=zyw \rangle.
\]
\end{itemize}
\end{example}

\end{document}